\documentclass[a4paper,12pt]{article}
\usepackage{amsmath}
\usepackage{amsthm}
\usepackage{graphicx, xcolor}
\usepackage{verbatim}
\usepackage{epsfig}
\usepackage{hyperref}
\usepackage{float}
\usepackage{color}
\usepackage{fullpage}
\usepackage{amsfonts}
\usepackage{amscd}
\usepackage{amssymb}
\usepackage{graphics}
\usepackage{caption}
\usepackage{subcaption}
\usepackage{amsmath}
\usepackage[english]{babel}
\usepackage{bbm}
\usepackage{yfonts}
\usepackage{mathrsfs}
\usepackage{color}
\usepackage[titletoc,toc,title]{appendix}
\DeclareFontFamily{OML}{rsfs}{\skewchar\font'177}
\DeclareFontShape{OML}{rsfs}{m}{n}{ <5> <6> rsfs5 <7> <8> <9>
rsfs7 <10> <10.95> <12> <14.4> <17.28> <20.74> <24.88> rsfs10 }{}
\DeclareMathAlphabet{\mathfs}{OML}{rsfs}{m}{n}
\newcommand{\BB}{{\mathbb{B}}}

\newcommand{\BE}{{\mathbb{E}}}

\newcommand{\BH}{{\mathbb{H}}}

\newcommand{\BN}{{\mathbb{N}}}
\newcommand{\BO}{{\mathbb{O}}}
\newcommand{\BP}{{\mathbb{P}}}

\newcommand{\BR}{{\mathbb{R}}}
\newcommand{\BS}{{\mathbb{S}}}

\newcommand{\CA}{{\mathcal{A}}}
\newcommand{\CB}{{\mathcal{B}}}
\newcommand{\CC}{{\mathcal{C}}}
\renewcommand{\CD}{{\mathcal{D}}}
\newcommand{\CE}{{\mathcal{E}}}

\newcommand{\CH}{{\mathcal{H}}}

\newcommand{\CV}{{\mathcal{V}}}

\newcommand{\CX}{{\mathcal{X}}}

\newcommand{\cc}{{\mathfrak c}}

\newcommand{\diam}{\textrm{diam}}

\newcommand{\scap}{\operatorname{cap}}
\newcommand{\dint}{\textrm{d}}

\newcommand{\arccot}{\operatorname{arccot}}

\newcommand{\Cr}{\operatorname{Cr}}
\newcommand{\Sep}{\operatorname{Sep}}
\newtheorem{theorem}{Theorem}[section]
\newtheorem{proposition}[theorem]{Proposition}
\newtheorem{lemma}[theorem]{Lemma}

\newtheorem{corollary}[theorem]{Corollary}

\newtheorem{remark}[theorem]{Remark}

\makeatletter
\let\@fnsymbol\@alph
\makeatother

\begin{document}
\numberwithin{equation}{section} \numberwithin{figure}{section}

\title{\bfseries Percolation of fat Poisson cylinders\\ in hyperbolic space}

\author{Carina Betken\footnotemark[1]\;, Erik Broman\footnotemark[2]\;, Anna Gusakova\footnotemark[3]\;, Christoph Th\"ale\footnotemark[4]}

\renewcommand{\thefootnote}{\fnsymbol{footnote}}

\footnotetext[1]{%
	Ruhr University Bochum and Hochschule für Gesundheit, Germany. Email: carina.betken@rub.de
}

\footnotetext[2]{%
Chalmers University of Technology and Gothenburg University, Sweden. \\
 Email: broman@chalmers.se
}

\footnotetext[3]{
	 University of M{\"u}nster, Germany. Email: gusakova@uni-muenster.de
}

\footnotetext[4]{%
	Ruhr University Bochum, Germany. Email: christoph.thaele@rub.de
}

\date{}
\maketitle
\begin{abstract}
In this paper we study Poisson processes of so-called "fat" cylinders
in hyperbolic space. As our main result we show that 
this model undergoes a percolation phase transition. We prove this 
by establishing a novel link between the fat Poisson cylinder 
process and semi-scale invariant random fractal models on 
the unit sphere and in $\BR^d.$ As a secondary result, we show that 
the semi-scale invariant fractal ball model in $\BR^d$ has 
a non-empty sheet phase.
\end{abstract}

\section{Introduction}

{Continuum percolation is the study of phase transitions and connectivity properties in random geometric structures defined by random points placed continuously in space and which are surrounded by random shapes. A typical example is the Poisson Boolean model, in which points are distributed according to a stationary Poisson point process, and each point is the centre of a ball or another geometric object that may overlap with others, resulting in the formation of random clusters. In Euclidean space, the analysis of such models has a long history, with fundamental issues centring around the existence of a critical intensity beyond which an infinite connected component emerges. The threshold phenomena in Euclidean continuum percolation have been studied extensively, and numerous results are known about the critical intensity, scaling behaviour near the threshold, and the geometry of the resulting infinite cluster. We refer to the monograph \cite{MeesterRoy} and mention exemplarily the articles \cite{AhlbergTassioTeiceira,DuminilRaouflTassio,GuereMarchand18,GuereLabey}. When extending these ideas and concepts to hyperbolic space, the geometry and curvature properties significantly influence percolation phenomena.  Unlike Euclidean space, hyperbolic space exhibits exponential volume growth, and this leads to different behaviours for the appearance of infinite clusters. The richer geometric environment of hyperbolic spaces affects critical thresholds and the geometry of percolation clusters in essential and sometimes surprising ways. We direct the reader to the papers \cite{BenjaminiSchramm,BenjaminiVisibility,Tykesson07} for details and further references.

The purpose of the present paper is to introduce the ``fat'' Poisson cylinder 
process and to investigate connectivity properties of this random set model in hyperbolic space.
In \cite{BromanTykessonCylHyperbolic} percolation properties of 
a related model -- consisting of what could 
now be referred to as ``thin'' Poisson cylinders -- were studied. What distinguishes such cylinder models from the more traditional Boolean model is that the geometric structure of cylinders induces long-range dependencies. Unlike the Boolean model, which typically considers balls or other bounded shapes whose overlap patterns are fairly localized, cylinder models are based on geometric shapes that extend over unbounded distances. This implies correlations between distant regions of space, as the presence of a cylinder in one region can influence connectivity and clustering properties far away. Although our primary result, Theorem \ref{thm:ConnectivityPhaseTransition},
which is about the connectivity properties of the fat Poisson cylinder model,
is similar to the one established in \cite{BromanTykessonCylHyperbolic} for the thin Poisson cylinder model, the methods and tools used here are fundamentally different. Our paper establishes a previously unexplored connection between the fat Poisson cylinder process and semi-scale invariant random fractals, and we leverage this relationship to prove our main result. In fact, the connection itself may be as significant as the primary findings, offering a new perspective that has the potential to be useful for other purposes as well. As an additional by-product we establish the existence of what one might call a ``sheet'' phase for the fractal model we consider.}

For $d\geq 1$ we denote by $\BH^d$ the $d$-dimensional hyperbolic space, which is the simply connected Riemannian manifold of constant sectional curvature $-1$, equipped with the hyperbolic metric $d_h$. Let $A_h(d,1)$ denote the space of all bi-infinite totally geodesic subspaces of dimension one, which we call hyperbolic lines in what follows, and let $L$ denote an element in $A_h(d,1).$ The space $A_h(d,1)$ carries a Haar measure $\mu^h$ which is invariant under all isometries of $\BH^d$. We refer to Section \ref{sec:HyperbolicGeometry} for background material on hyperbolic geometry.  Each hyperbolic line $L\in A_h(d,1)$ is parametrized by the totally geodesic hypersurface $H$ orthogonal to $L$ and passing through a fixed point $o\in\BH^d$, and a point $x\in H$. Using this parametrization, for a fixed hyperbolic line $L=L(H,x)\in A_h(d,1)$ we consider the 
set $\hat A_L$ of all hyperbolic lines orthogonal to $H$ and whose 
intersection with $H$ has hyperbolic distance at most 
$1$ from $x$. That is,
\begin{equation}\label{eq:DefHatAL}
\hat A_L := \{L'\in A_h(d,1) : L'\perp H,d_h(L'\cap H,x)< 1\}.
\end{equation} 
The (open) 
cylinder $\cc_L$ associated with $L$ is then defined to be the union 
of all elements in  $\hat A_L$
\begin{equation*}
\cc_L := \bigcup_{L'\in\hat A_L}L',
\end{equation*}
see Figure \ref{fig:FatCyl}. 

\begin{figure}[t]
\centering	
\includegraphics[width=0.5\columnwidth]{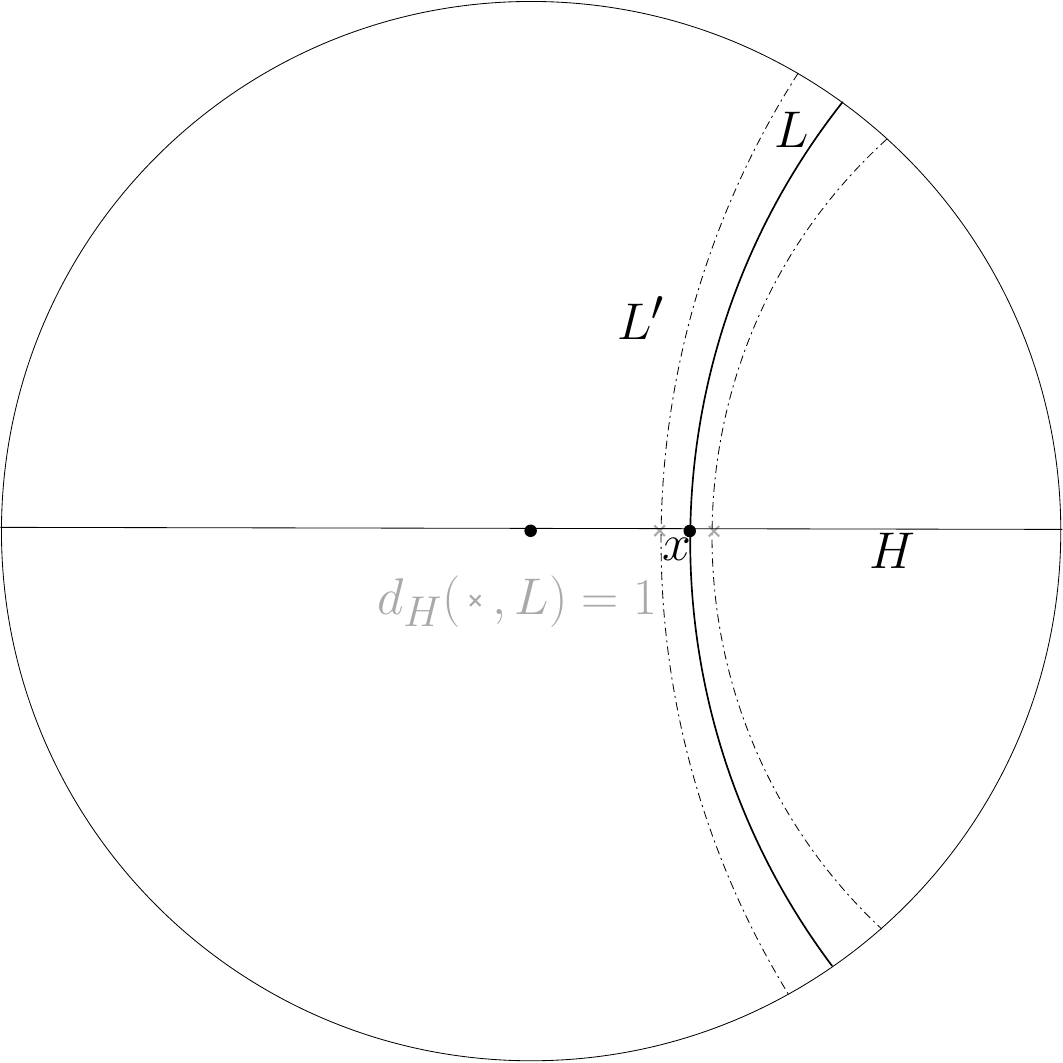}
\caption{\small Construction of a (fat) hyperbolic cylinder in the Poincar\'e ball model of $\mathbb{H}^d$.}
\label{fig:FatCyl}
\end{figure}

For $\lambda>0$ let $\Psi(\lambda \mu^h)$ be a Poisson  
process on the space $A_h(d,1)$ with intensity measure $\lambda\mu^h$ 
and consider the random set
\begin{equation} \label{eqn:Cdef}
\CC(\lambda\mu^h) := \bigcup_{L\in\Psi(\lambda \mu^h)}\cc_L,
\end{equation}
consisting of all cylinders associated with the hyperbolic lines
from $\Psi(\lambda \mu^h)$. In this paper we deal with
the connectivity properties of the set $\CC(\lambda\mu^h)$ 
for which we prove the following result. 

\begin{theorem}\label{thm:ConnectivityPhaseTransition}
There exists $\lambda_c\in(0,\infty)$ such that the following holds: 
\begin{itemize}
\item[(i)] if $\lambda<\lambda_c$ then $\CC(\lambda\mu^h)$ is not connected 
with positive probability;
\item[(ii)] if $\lambda>\lambda_c$ then 
$\CC(\lambda\mu^h)$ is connected with probability one.
\end{itemize}
\end{theorem}

\begin{remark}\label{rem:1} We would like to stress at this point that the result of 
part (i) of Theorem \ref{thm:ConnectivityPhaseTransition} is best 
possible 
and cannot be improved to a probability one statement. A further 
discussion needs more background and we therefore postpone it to Section 
\ref{sec:inducedBool}.
\end{remark}

As mentioned before, a similar type of result was established in
\cite{BromanTykessonCylHyperbolic} for the different notion of a hyperbolic cylinder. More precisely, a cylinder corresponding to 
a hyperbolic line $L\in A_h(d,1)$ was defined as the set of all points 
in $\BH^d$ having hyperbolic distance at most $1$ from $L$. In 
particular, such a cylinder has a single endpoint at infinity,
while in the current paper, the cylinders have a non-trivial intersection with the Gromov (ideal) boundary of $\mathbb{H}^d$ (compare Figure \ref{fig:FatCyl} to Figure \ref{fig:ThinCyl},
and see further Lemma \ref{lm:scaps}). 
\begin{figure}[t]
\centering	
\includegraphics[width=0.5\columnwidth]{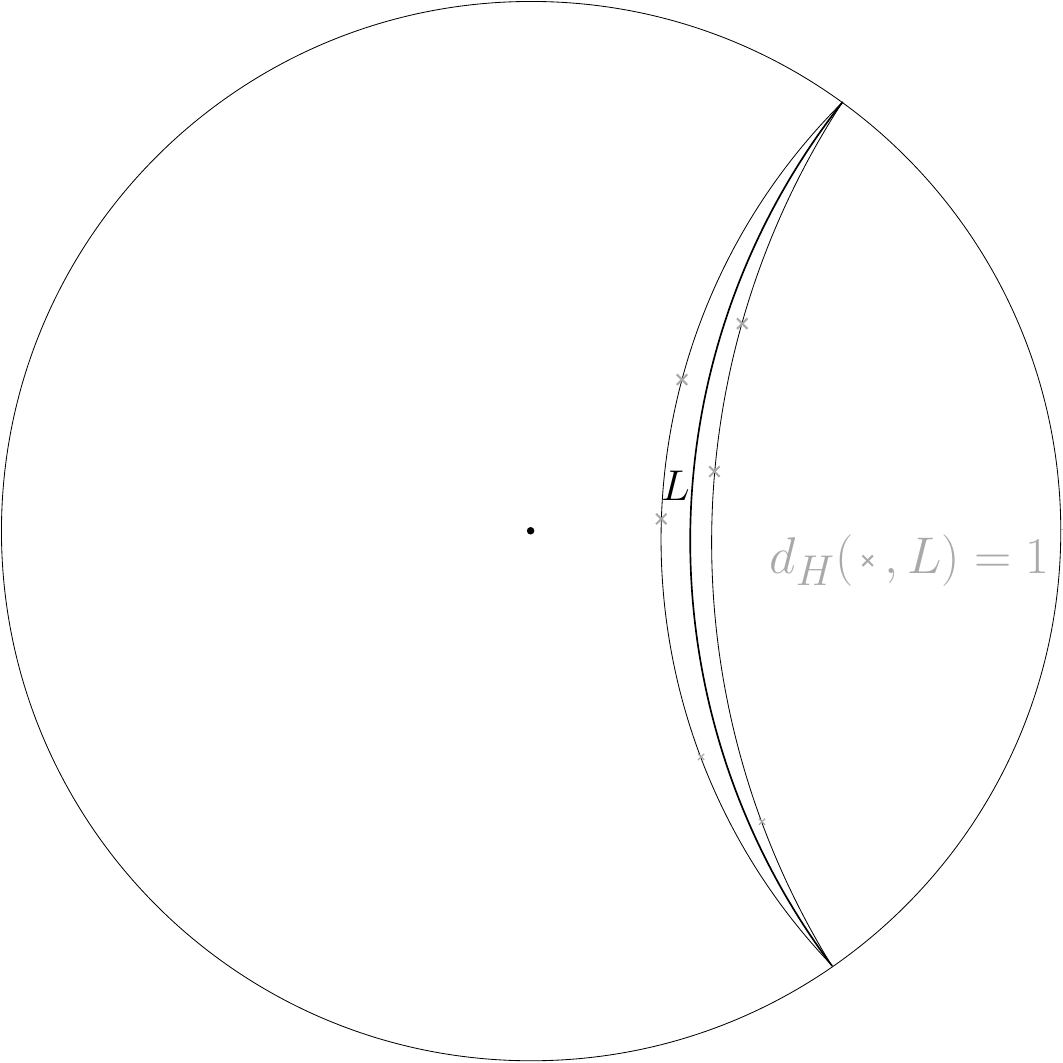}
\caption{\small Construction of a (thin) hyperbolic cylinder in the Poincar\'e ball model of $\mathbb{H}^d$.}
\label{fig:ThinCyl}
\end{figure}
This is why we casually refer to the cylinders in the current paper 
as ``fat'' while the ones in \cite{BromanTykessonCylHyperbolic} 
can be called ``thin''. From now on, whenever a cylinder is mentioned
it will be understood to be a fat cylinder. The comparison to \cite{BromanTykessonCylHyperbolic} also shows that 
the concept of a cylinder in Euclidean space has several different 
generalizations in the hyperbolic set-up. 

We note that while our main result, Theorem 
\ref{thm:ConnectivityPhaseTransition}, is the analogue of the main result 
in \cite{BromanTykessonCylHyperbolic}, the methods used in the present paper are 
completely different. Indeed, in \cite{BromanTykessonCylHyperbolic}
a coupling with an infinite offspring/infinite type branching process
was used in order to prove the result corresponding to $(i)$ of 
Theorem \ref{thm:ConnectivityPhaseTransition}, while the analogue
of $(ii)$ was proven by coupling the cylinder process to a 
hyperbolic ball process. In this paper, both parts, 
i.e. $(i)$ and 
$(ii)$, are instead proven by studying a semi-scale invariant process
on the boundary of the $d$-dimensional open Euclidean unit ball $\BB^d$ induced by the cylinders. Therefore, our
proofs here rely on results for 
the study of random fractals (see in particular \cite{Broman} 
and \cite{O96}). 

We also point out the important fact that part 
$(ii)$ of Theorem \ref{thm:ConnectivityPhaseTransition} can be alternatively derived
by coupling the process of "fat" and "thin" cylinders and applying the result from \cite{BromanTykessonCylHyperbolic}. 
However, we choose to go through 
a different route for two reasons. Firstly, our paper will be 
more self-contained. Secondly, one aim of the paper is to demonstrate 
the connection between cylinder processes in hyperbolic space and 
random fractal models on the boundary of $\BB^d$ and using results from 
\cite{BromanTykessonCylHyperbolic} will not establish such a connection. 
In addition, our proof of part $(ii)$ is relatively short, while
the proof of part $(i)$ will require the most of our 
efforts.

In our proof of part $(ii)$ of Theorem 
\ref{thm:ConnectivityPhaseTransition}, we will study the so-called
fractal ball model. As a by-product, we 
will prove a statement that may be of independent interest and
in order to state this secondary result, we will first have to  introduce the model. For simplicity, we do this  in the $d$-dimensional Euclidean space $\BR^d$, although it will later be applied to $\BR^{d-1}$.  In order to not make the introduction
too long, we will postpone further motivation and background until 
Section \ref{sec:Notation}. We consider here a
Poisson  process $\Psi(\lambda \mu^E)$ on 
$\BR^d \times (0,1]$ 
with intensity measure $\lambda \mu^E$ where
\begin{equation} \label{eqn:muEdef}
\mu^E(\dint (x,r))= \ell_{d}(\dint x)\times {\bf 1}(0<r\leq 1) 
r^{-(d+1)} \dint r,
\end{equation}
and where $\ell_d$ denotes Lebesgue measure in $\BR^d.$ The process $\Psi(\lambda \mu^E)$ gives 
rise to a semi-scale invariant collection of balls
(see Section \ref{sec:Notationfracball}).
Namely, to any pair $(x,r)\in \Psi(\lambda \mu^E)$ we associate the open Euclidean ball 
$B(x,r)$ of radius $r>0$ centered at $x$. Then, we let 
\begin{equation} \label{eqn:Cdeffracball}
\CC(\lambda \mu^E)=\bigcup_{(x,r)\in \Psi(\lambda\mu^E)} B(x,r),
\end{equation}
and define
\begin{equation} \label{eqn:lambdasheetdef}
\lambda_{\rm sheet}:=\sup\{\lambda>0: \BP(\CC(\lambda\mu^E)
\textrm{ contains only bounded components})=1\}.
\end{equation}
Observe, that if $\CC(\lambda\mu^E)$ contains only bounded components 
then every such component must be surrounded by a subset of the 
complement which we can think of as a ``sheet''. We also note that 
the probability in the definition of $\lambda_{\rm sheet}$ is either 0 
or 1, which follows by ergodicity. 
It is not apriori clear that $\lambda_{\rm sheet}>0,$ but our secondary 
result establishes this.
\begin{theorem} \label{thm:Cbounded}
We have that $\lambda_{\rm sheet}>0.$
\end{theorem}

The remaining parts of this paper are structured as follows. In Section 
\ref{sec:Preliminaries} we introduce basic notation and give an outline of 
the strategy of the proof of Theorem \ref{thm:ConnectivityPhaseTransition}.
This section also contains a short discussion about the fractal ball model.
In Section \ref{sec:inducedBool} we provide a detailed description of the process 
of caps on the unit sphere induced by the (fat) Poisson cylinder model. Furthermore, 
in Section \ref{sec:inducedBool} we state two auxiliary theorems 
(Theorems \ref{thm:cover} and \ref{thm:disconnect}) which will imply Theorem 
\ref{thm:ConnectivityPhaseTransition}. The induced
process of caps is further analyzed in Section \ref{sec:Boolcaps}, while
a more general version of Theorem \ref{thm:cover} is proved in Section \ref{sec:Proof1}.
Section \ref{sec:fracball} is dedicated to proving Theorem \ref{thm:Cbounded} and is 
used in Section \ref{sec:Proof2} in order to prove Theorem \ref{thm:disconnect}.

\section{Preliminaries} \label{sec:Preliminaries}

\subsection{Notation}\label{sec:Notation}

Let $\langle \,\cdot\,,\,\cdot\,\rangle$ denote the scalar product 
and $\|\,\cdot\,\|$ the Euclidean norm in $\BR^d$. By 
$B(x,r) \subset \BR^d$ we denote an open Euclidean ball centered 
at $x$ and with radius $r>0$.  When the dimension $d$ of a ball will become important we will emphasize it by writing $B^d(x,r)$. Similarly, while we write $\Vert x \Vert$ for the 
usual Euclidean norm in $\BR^d,$ we will write $\Vert x \Vert_{d-1}$
in order to emphasize that we are dealing with the Euclidean norm of a point 
$x\in\BR^{d-1}.$
Further, for $ d \in \mathbb{N}, $ denote by
$$
\omega_d := {2\pi^{d/2}\over\Gamma({d\over 2})}
$$
the surface area of the $(d-1)$-dimensional Euclidean unit sphere $\BS^{d-1}$. By $\sigma_{d-1}$ we denote the $(d-1)$-dimensional 
spherical Lebesgue measure on $\BS^{d-1}$ so that $\sigma_{d-1}(\BS^{d-1})=\omega_d$. Given a finite set $A$ we denote by $|A|$ its cardinality. For a given subset $B\subset \BR^d$ and $0<a<\infty$ we write $aB:=\{ax\colon x\in B\}$.

Consider a metric space 
$(\BS^{d-1}, d_{s})$, 
where 
\begin{align}\label{eq:SphericalDistance}
\cos d_s(u,v) = \langle u,v\rangle,\qquad u,v\in\BS^{d-1}.
\end{align}
We let
$$
\scap(u,\alpha):=\{v\in\BS^{d-1}\colon d_s(u,v)< \alpha\}\subset\BS^{d-1}
$$ 
be an open spherical cap with center $x\in\BS^{d-1}$ and (spherical)
radius $\alpha\in(0,\pi]$. We use  two different 
projections. Firstly, we let $\Pi_{\rm rad}(x)$ 
denote the radial projection that maps a point 
$x\in \BR^d\setminus \{o\}$ onto ${x\over\|x\|}\in\BS^{d-1}$. 
Secondly, we define the vertical projection 
$\Pi_{\rm vert}: \BS^{d-1}_{+} \to B^{d-1}(o,1)$ 
which maps a point 
$x\in \BS^{d-1}_+
:=\{(x_1,\ldots, x_d): x_1^2+\ldots+x_d^2=1, x_d >0\}$
in the upper hemisphere of $\BS^{d-1}$
onto the point $(x_1,\ldots, x_{d-1})\in B^{d-1}(o,1).$ 
Note the restricted domain of this projection, which ensures that 
it is injective so that
the inverse $\Pi^{-1}_{\rm vert}(y)$ is unique
for every $y\in B^{d-1}(o,1).$

\subsection{Hyperbolic geometry}\label{sec:HyperbolicGeometry}

In this paper we work with the Poincar\'e ball model for the $d$-dimensional hyperbolic space $\BH^d$, that is, we identify $\BH^d$ with the open Euclidean unit ball $\BB^d$ together with the Poincar\'e metric given by
\begin{equation}\label{eq:HypMetric}
	\cosh d_h(x,y) := 1+{2\|x-y\|^2\over (1-\|x\|^2)(1-\|y\|^2)},\qquad x,y\in\BB^d.
\end{equation}
Recall that by $A_{h}(d,1)$ we denote the space of  one dimensional totally geodesic subspace of $\BH^d$. In the Poincar\'e model a hyperbolic line $L\in A_h(d,1)$, is identified with the intersection of a Euclidean circle $\BS^1(L)$, orthogonal to $\BS^{d-1}$, with the unit ball $\BB^{d}$. In the case when $o\in L$ we have that $\BS^1(L)$ is a circle of infinite radius, namely a Euclidean line.  Further by $G_h(d,d-1)$ we denote the hyperbolic Grassmannian
of totally geodesic hypersurfaces of $\BH^d$ passing 
through the centre $o$ of $\BB^d$, and let $H$ denote an element in $G_h(d,d-1).$ In the Poincar\'e ball model any $H\in G_h(d,d-1)$ is identified with an intersection of Euclidean linear $(d-1)$-dimensional subspace with $\BB^d$.
Let $u_H\in \BS^{d-1}$
be one of two possible normal unit vectors 
of $H$, and observe that $\pm u_H$ both correspond to the same $H\in G_h(d,d-1).$ 
In this paper, it is most of the time only $H$ which is relevant, but later 
(in particular Section \ref{sec:inducedBool}), the choice of normal vector becomes pertinent.
Therefore, when we write $H\in G_h(d,d-1)$ we are really considering the pair
$(H,u_H),$ i.e. $H$ along with its orientation, but we only stress this when the choice 
of $u_H$ is important. 
We then use uniform measure $\omega_{d}^{-1}\sigma_{d-1}(\dint u_H)$ on $\BS^{d-1}$ in order to construct an isometry invariant measure $\mu^h$ on the space $A_h(d,1)$. In order to 
introduce this measure we recall that an element $L\in A_h(d,1)$ 
is uniquely determined by its orthogonal hypersurface $H\in G_h(d,d-1)$ 
%(that passes through $o$) 
and the intersection point 
$x=L\cap H$. For given $H\in G_h(d,d-1)$ and $x\in H$ we write $L=L(H,x)\in A_h(d,1)$ for the corresponding hyperbolic line.
Using this parametrization we can express the invariant measure $\mu^h$ on $A_h(d,1)$ as
\begin{equation}\label{eq_plane_intensity}
	\mu^h(\,\cdot\,) := \omega_{d}^{-1}\int_{\BS^{d-1}}\int_H\cosh(d_h(x,o))\,{\bf 1}\{L(H,x)\in\,\cdot\,\}\,\CH^{d-1}(\dint x)\,\sigma_{d-1}(\dint u_H),
\end{equation} 
where $\CH^{d-1}$ denotes the
$(d-1)$-dimensional Hausdorff measure with respect to the hyperbolic structure, see \cite[Equation (17.41)]{Santalo}.

\subsection{Outline of the proof}

The proof of Theorem \ref{thm:ConnectivityPhaseTransition} will proceed
through a number of steps. We will start by noting that the hyperbolic Poisson cylinder process considered in the Poincar\'e ball model can be naturally identified with the Poisson process of spherical caps (see Section \ref{sec:inducedBool}). As a next step we will then analyze this cap process using known results about scale invariant processes in $\BR^d$. 

Since throughout the paper we will work with different processes living in $\BH^d$, $\BS^{d-1}$ or $\BR^{d-1}$ we will use the following conventions, where the superscripts $h$, $s$ and 
$E$ indicate whether we are working in hyperbolic, spherical or Euclidean space.

\begin{enumerate}
\item Our starting point is a Poisson  processes of hyperbolic
lines in hyperbolic space, i.e. on $A_h(d,1)$, with intensity measure $\lambda \mu^h$, $\lambda>0$. Such a process will be denoted by 
$\Psi(\lambda \mu^h)$. The 
corresponding covered region will be $\CC(\lambda \mu^h)\subset \BH^d$, see \eqref{eqn:Cdef}.

\item The Poisson  process of cylinders will induce several Poisson 
point processes on $\BS^{d-1}\times (0,c]$ (where $0<c\leq \pi$ is some 
number). Such processes will be denoted by $\Psi(\lambda \mu^s)$, 
where $\mu^s$ is some measure on $\BS^{d-1}\times (0,c]$, which will be defined later. 
Here, a pair $(u,\alpha)$ 
corresponds to an open spherical cap $\scap(u,\alpha)$, and the 
covered region is given by 
\[
\CC(\lambda \mu^s)=\bigcup_{(u,\alpha)\in \Psi(\lambda \mu^s)}
\scap(u,\alpha)\subset \BS^{d-1}.
\]

\item In order to analyze the random set $\CC(\lambda \mu^s),$ it will 
be necessary to 
project it vertically to $\BR^{d-1}.$ Therefore, we shall consider 
Poisson  processes $\Psi(\lambda \mu^E)$ on $\BR^{d-1}\times (0,\tilde c]$ for some constant $0<\tilde c<\infty$. Here, $\mu^E$ is some measure on $\BR^{d-1}\times (0,\tilde c]$, with
which we will have to deal in several different versions 
(for instance the one given by \eqref{eqn:muEdef}).
The corresponding covered region is given by
\[
\CC(\lambda \mu^E)=\bigcup_{(x,r)\in \Psi(\lambda \mu^E)}
B^{d-1}(x,r)\subset \BR^{d-1},
\] 
and the vacant region is given by 
$\CV(\lambda \mu^E)=\BR^{d-1}\setminus \CC(\lambda \mu^E)$. 
We note that in one case, the open ball $B^{d-1}(x,r)$ will be 
replaced by a distorted version. 
\end{enumerate}

The Poisson  processes $\Psi(\lambda \mu^h),\Psi(\lambda \mu^s)$ and 
$\Psi(\lambda \mu^E)$ are formally random counting measures, but it will be 
convenient for us to think of them as random countable collections of corresponding  elements (sets), which is a 
standard abuse of notation. For instance, given a measurable enumeration of the lines
in the process (see \cite{LP18} for more details), we will write 
$\Psi(\lambda \mu^h)=\{L_1,L_2,\ldots\}$ instead of 
$\Psi(\lambda \mu^h)=\sum_{i\geq 1} \delta_{L_i}$, where 
$\delta_{L_i}$ denotes Dirac measure on $L_i$.

\subsection{The fractal ball model} \label{sec:Notationfracball}
In this subsection we will  briefly describe the fractal ball model in 
order to further motivate our secondary result, Theorem \ref{thm:Cbounded}.

It is easy to prove 
(see \cite{Broman}) that $\Psi(\lambda \mu^E)$, where $\mu^E$ is given 
by \eqref{eqn:muEdef}, is a semi-scale invariant 
collection of balls in $\BR^{d}$ in the following sense. If $0<a<1$ and 
\[
\tilde{\Psi}_a:=\{G/a: G\in \Psi(\lambda \mu^E), \diam(G/a)\leq 2\},
\]
then $\tilde{\Psi}_a$ has the same distribution as $\Psi(\lambda \mu^E)$, where we use the convention that a pair $G=(x,r)\in\BR^d\times(0,\infty)$ is identified with the Euclidean ball $B(x,r)$ with radius $r$ centred at $x$.
Intuitively, this means that if we scale space by a factor of $1/a$ and 
then remove any ball that has become ``too large'' (i.e. diameter 
larger than 2),
we will not be able to tell the difference between the distributions of the random sets 
before and after scaling and deletion. In this context, the object 
that is typically studied is the complement set of $\CC(\lambda\mu^E),$
i.e.
\begin{equation} \label{eqn:defVfracball}
\CV(\lambda\mu^E):=\BR^d \setminus \CC(\lambda\mu^E). 
\end{equation}
We see that $\CV(\lambda\mu^E)$ is a random fractal set. 
Indeed, 
it follows
from the semi-scale invariance that for any fixed $x\in \BR^d,$ we have
that $\BP(x\in \CV(\lambda\mu^E))=0.$
Furthermore, the set $\CV(\lambda\mu^E)$ 
undergoes several phase transitions as the parameter $\lambda>0$ varies,
as we now briefly explain. First, there exists $\lambda_c<\infty$ such 
that for any $0<\lambda\leq \lambda_c$ (where the equality is non-trivial 
and was established in \cite{BC}), $\CV(\lambda\mu^E)$ contains
connected components larger than one point almost surely. This is 
called the connected phase
since $\CV(\lambda\mu^E)$ contains macroscopic connected components. Secondly, 
there exists $\lambda_c<\lambda_e<\infty$ such that for 
$\lambda_c<\lambda<\lambda_e,$ $\CV(\lambda\mu^E)$ is non-empty 
but does not contain any connected components larger than one point 
almost surely. This means that $\CV(\lambda\mu^E)$ is totally 
disconnected but non-empty. This phase is referred to as the 
dust phase. 
For $\lambda \geq \lambda_e$ (the equality is again non-trivial and was 
established in \cite{Broman}), we have that $\CV(\lambda\mu^E)=\emptyset$
with probability one. The definitions of these phases are natural, but 
one can also consider whether other phases exists. For instance, 
while having received considerably less attention (the only result
of this sort is \cite{O96}), one can ask whether there exists 
a non-empty ``sheet''-phase. Theorem 
\ref{thm:Cbounded} shows that it does, which is why we believe that 
this theorem may be of independent interest. We end this 
discussion by observing that while it is clear that 
$\lambda_{\rm sheet}\leq \lambda_c$, since a sheet is a connected component,
it would be desirable to establish whether 
$\lambda_{\rm sheet}<\lambda_c.$ However, since this is not relevant for the 
current paper, we leave this problem for future research.

\section{The induced Boolean model on $\BS^{d-1}$} \label{sec:inducedBool}

\begin{figure}[t] 
	\begin{center}
		\includegraphics[scale=0.5]{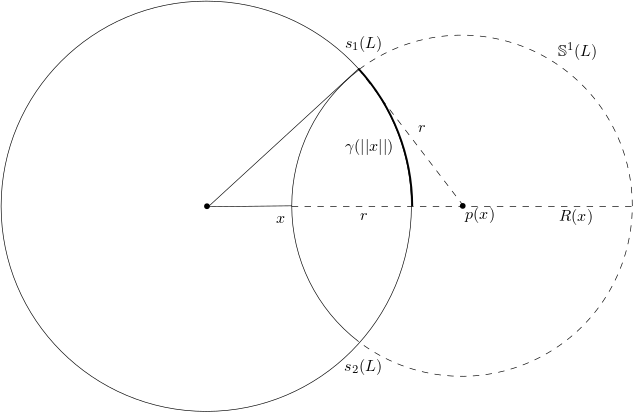}
	\end{center}
	\caption{\small Derivation of the identity for $ \gamma(\|x\|)$.}
	\label{Fig:ArcLengthProjection}
\end{figure}

In this section we will describe the connection between the cylinder process 
$\Psi(\lambda\mu^h)$ in $\BH^d$ and a Poisson  process of caps on the unit sphere $\BS^{d-1}$. Let $L=L(H,x)\in A_h(d,1)$ with $x\neq o$ be a hyperbolic line, which in the Poincar\'e model we identify with $\BS^1(L)\cap \BB^d$, where $\BS^1(L)$ is a circle orthogonal to $\BS^{d-1}$. Let $2\gamma(\|x\|)$ denote the spherical length of the radial projection $\Pi_{\rm rad}(\BS^1(L)\cap \BB^d)$ of $\BS^1(L)\cap \BB^d$ onto the unit sphere $\BS^{d-1}$ (see also Figure \ref{Fig:ArcLengthProjection}).
Note that the length of this projection only depends on the distance of $L$ to the origin. In the case $x=o$ we set $\gamma(0):={\pi\over 2}$, which clearly 
corresponds to the limit of $\gamma(\|x\|)$ as $\|x\| \to 0.$
It holds that
\begin{equation}\label{eq_alpha}
\gamma(\|x\|)=\arctan\Big({1-\|x\|^2\over 2\|x\|}\Big)\in \Big[0,{\pi\over 2}\Big].
\end{equation}
Indeed, let $s_1(L),s_2(L)\in\BS^{d-1}$ denote the points of intersection of the Euclidean circle $\BS^{1}(L)$ with the unit sphere 
$\BS^{d-1}$ (see again Figure \ref{Fig:ArcLengthProjection}). We will explain below, which of the intersection points 
is chosen to be $s_1(L)$ (respectively, $s_2(L)$), but at the moment it does not play a role.
In the case $L\in G_h(d,1),$ i.e.\ when $x=o,$ the circle $\BS^{1}(L)$ becomes a line, which can be considered as a circle passing through infinity. For $x\neq o$ let $R(x)$ denote the ray from $o$ through the point $x$. Then the hyperplane tangent to $\BS^{d-1}$ at the point $s_1(L)$ intersects the ray $R(x)$ in exactly one point $p(x)$, which can equivalently be described as the image of $x$ under inversion at $\BB^d$. Moreover, since $\BS^{1}(L)$ and $\BS^{d-1}$ intersect orthogonally, $p(x)$ is the center of the circle $\BS^{1}(L)$. Hence $r:=\|s_1(L)-p(x)\|=\|p(x)-x\|$ and the triangle with vertices $s_1(L)$, $p(x)$, $o$ is a right-angled triangle. Thus,
$$
\gamma(\|x\|)=\arctan(r),
$$
with $r^2+1=(r+\|x\|)^2$, or equivalently, $r= \frac{1-\|x\|^2}{2\|x\|}$, which gives \eqref{eq_alpha}. Having established \eqref{eq_alpha},
we note further that it follows from \eqref{eq:HypMetric} that
\begin{equation}\label{eq:EuclideanHyperbolicDist}
\|x\|=\sqrt{\frac{\cosh d_h(o,x)-1}{ \cosh d_h(o,x)+1}},
\end{equation}
and so we may conclude that
\begin{equation}\label{eq_transform}
\cot \gamma(\|x\|)=\sinh d_h(o,x).
\end{equation}
Furthermore, note that every hyperbolic line $L=L(H,x)\in A_h(d,1)$ is uniquely determined by 
the points $s_1(L),s_2(L)\in \BS^{d-1}$. 
When $x\neq o$ we may set 
\begin{equation}\label{eq_s}
\begin{array}{l}
s_{1}(L)=\cos\gamma(\|x\|)\Pi_{\rm rad}(x)+ \sin\gamma(\|x\|)u_{H},\\
s_{2}(L)=\cos\gamma(\|x\|)\Pi_{\rm rad}(x)- \sin\gamma(\|x\|)u_{H}
\end{array}
\end{equation}
where, as explained in connection to \eqref{eq_plane_intensity}, 
$u_{H}$ denotes the unit normal vector which defines $H,$ 
 see Figure~\ref{Fig:IntersectionPoints} for an illustration. 
At this point the choice between $u_H$ and $-u_H$ becomes relevant, see further Remark 
\ref{remark:uH} below.
\begin{figure}[t]
	\begin{center}
		\includegraphics[scale=0.6]{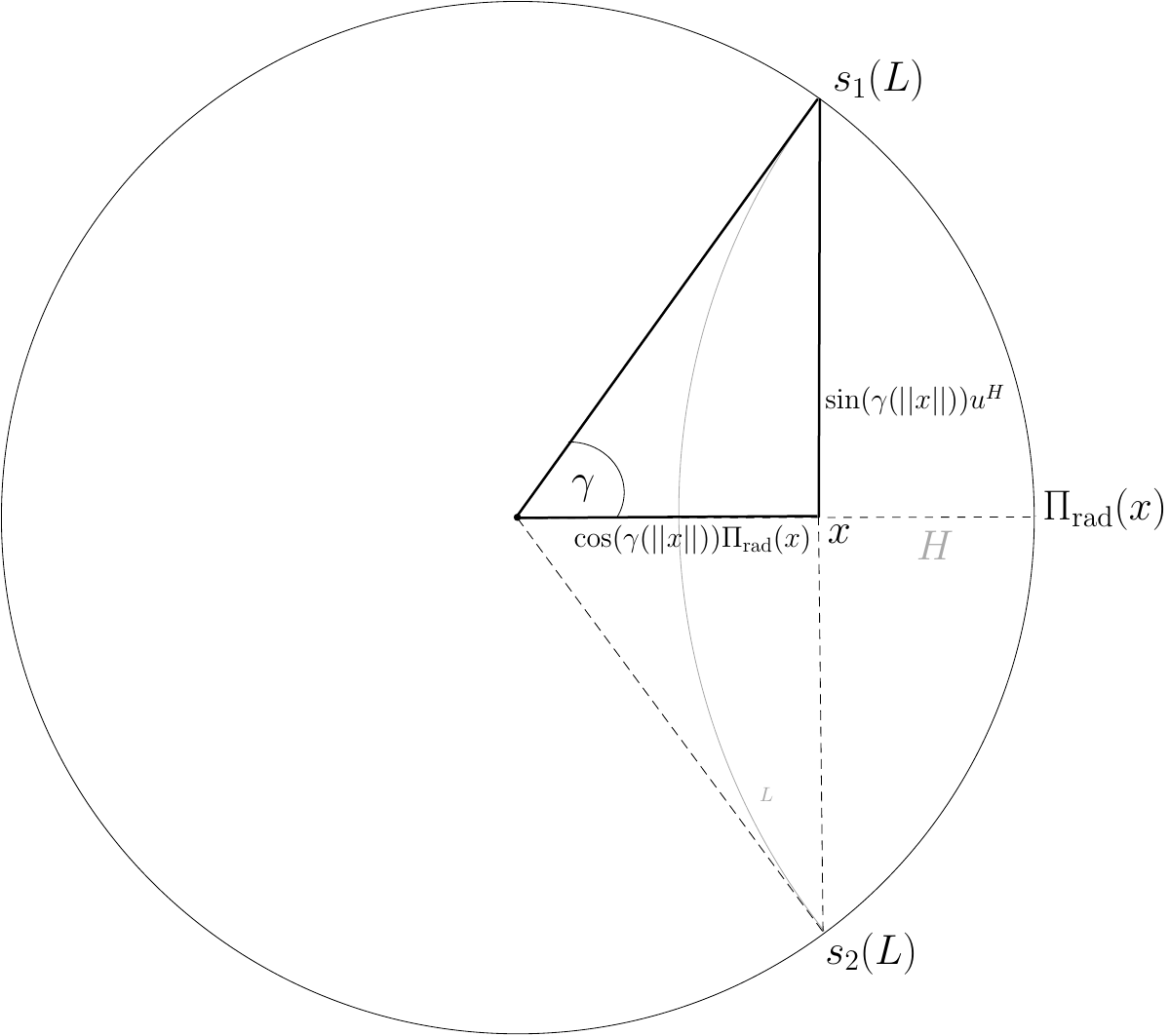}
	\end{center}
	\caption{\small Illustration supporting the derivation of \eqref{eq_s}.}
	\label{Fig:IntersectionPoints}
\end{figure}
In the case $x=o$ we set
$$
s_{1,2}(L)=\pm u_{H}.
$$
In what follows we will use formula \eqref{eq_s} also in the degenerate case $x=o$ by defining $\Pi_{\rm rad}(o):=(1,0,\ldots,0)$.

Let us now consider a cylinder $\cc_L$. In the next lemma we will show that 
the set $ \{s_1(L')\colon L'\in \hat{A}_L\} \subset \BS^{d-1} $, where $\hat{A}_L$ is given by \eqref{eq:DefHatAL}, is an open spherical 
cap and prove a  few useful properties. %of $c(L)$.

\begin{lemma}\label{lm:scaps}
For any $L=L(H,x)\in A_h(d,1)$ the set 
$ \{s_1(L')\colon L'\in \hat{A}_L\} \subset \BS^{d-1} $
 is an open spherical cap 
$\scap(z(L),\beta(\Vert x \Vert))$ with center $z(L)$ and spherical radius $\beta(\|x\|)\in (0,\arccos(1/\cosh(1))]$. Moreover, for $x\neq o$ we have
\begin{equation}\label{eq_center_cap}
z(L)=\cos(\gamma(a(\|x\|)\|x\|))\Pi_{\rm rad}(x)+ \sin(\gamma(a(\|x\|)\|x\|))u_{H},
\end{equation}
and 
$$
\beta(\Vert x \Vert)=\arccos\left({a(\|x\|)(1+\|x\|^2)\over 1+a(\|x\|)^2\|x\|^2}\right),
$$
where the function $a(t)$ has the form
\begin{equation}\label{eq_12_06_2}
a(t)={\sqrt{\cosh^2(1)(1-t^2)^2+4t^2}-\cosh(1)(1-t^2)\over 2t^2}\in (0,1).
\end{equation}
If $x=o$ we have $z(L)=u_{H}$ and $\beta(0)=\arccos(1/\cosh(1))$.
\end{lemma}

\begin{proof}
	We start by introducing some notation. For $z\in \mathbb{R}^d\backslash\{o\}$ we set $u_z:=\Pi_{\rm rad}(z)=z/\|z\|,$ 
	and let $u_o:=u_{H}$ given that $H\in G_{h}(d,d-1)$ and the normal vector $u_H$
	are fixed.
Now, for any $L' =L(H,y) \in \hat{A}_L$ we have that $y= y(L'):=L'\cap H $ satisfies $ d_h(x,y)<1 $. Let $ z(L) =\cos(\gamma(\|ax\|))\Pi_{\rm rad}(x)+ \sin(\gamma(\|ax\|))u_{H} $ for some $ a \ge 0 $. Then by \eqref{eq_s} and the definition of the spherical distance \eqref{eq:SphericalDistance} we have
\begin{align*}
\cos \big(d_s(s_1(L'),z(L))\big)&=\cos\gamma(\|y\|)\cos\gamma(\|ax\|)\langle u_x,u_y\rangle +\sin\gamma(\|y\|)\sin\gamma(\|ax\|),
\end{align*}
where we used that $u_x, u_y \in H$ and thus $ \langle u_x,u_{H}\rangle=\langle u_y,u_{H}\rangle =0$. Now, since
$\cos(\arctan(z))=\frac{1}{\sqrt{1+z^2}}$ and $\sin(\arctan(z))=\frac{z}{\sqrt{1+z^2}}$ for $z\in\BR$, \eqref{eq_alpha} yields
$$
\cos\gamma(\|x\|)={2\|x\|\over 1+\|x\|^2}\qquad\text{and}\qquad\sin\gamma(\|x\| )={1-\|x\|^2\over 1+\|x\|^2},
$$
which in turn gives
\begin{equation}\label{eq:CosSphericalDistance}
	\cos \big(d_s(s_1(L'),z(L))\big)={4a\|x\|\|y\|\langle u_x,u_y\rangle+(1-a^2\|x\|^2)(1-\|y\|^2)\over (1+\|y\|^2)(1+a^2\|x\|^2)}.
\end{equation}
Since $0\le d_h(x,y)< 1$ and $\cosh(z)$ is a strictly increasing function for 
positive $z\in\BR$ we have from \eqref{eq:HypMetric} that
$$
1+{2\|x-y\|^2\over (1-\|x\|^2)(1-\|y\|^2)}=\cosh d_h(x,y) < \cosh(1).
$$
From this we can deduce the inequality
\[
	\langle x-y,x-y \rangle< \frac{\cosh(1)-1}{2}(1-\|x\|^2)(1-\|y\|^2),
	\]
	which in turn is equivalent to
	\begin{equation}\label{eq:Equation3}
	2 \langle x,y \rangle > -  \frac{\cosh(1)-1}{2}(1-\|x\|^2)(1-\|y\|^2)+ \|x\|^2+ \|y\|^2.
\end{equation}
Further, since $2\langle x,y\rangle = 2\|x\|\|y\|\langle u_x,u_y\rangle$ we obtain
\begin{equation}\label{eq_12_06_1}
2\|x\|\|y\|\langle u_x,u_y\rangle > \big(1+c(1-\|x\|^2)\big)\|y\|^2+\|x\|^2-c(1-\|x\|^2)
\end{equation}
for $c:=(\cosh(1)-1)/2$.
Using this inequality in \eqref{eq:CosSphericalDistance} we see that
\begin{align}\label{eq:CosDistance}
\cos \big(d_s(s_1(L'),z(L))\big)&> {\big(2a(1+c(1-\|x\|^2))-1+a^2\|x\|^2\big)\|y\|^2\over (1+\|y\|^2)(1+a^2\|x\|^2)}\notag\\
&\qquad\qquad +{(2\|x\|^2-2c(1-\|x\|^2))a+1-a^2\|x\|^2\over (1+\|y\|^2)(1+a^2\|x\|^2)}.
\end{align}
 Moreover, equality holds if and only if $d_h(x,y)=1$. Next we show that for $a=a(x)$ we have $ d_s(z(L), s_1(L(H,y_1)))=d_s(z(L), s_1(L(H,y_2))) $ for all $ y_1, y_2$ with $d_h(x,y_i)=1, i=1,2$. Due to \eqref{eq:CosDistance}  the latter is equivalent to
 \begin{align}\label{eq:MaxDistanceCenter}
 	{c_1(a,\|x\|)\|y_1\|^2+c_2(a,\|x\|)\over (1+\|y_1\|^2)(1+a^2\|x\|^2)}=	{c_1(a,\|x\|)\|y_2\|^2+c_2(a,\|x\|)\over (1+\|y_2\|^2)(1+a^2\|x\|^2)},
 \end{align}
 where  $ c_1(a,\|x\|)=2a(1+c(1-\|x\|^2))-1+a^2\|x\|^2 $ and $ c_2(a,\|x\|)=(2\|x\|^2-2c(1-\|x\|^2))a+1-a^2\|x\|^2$. Equation \eqref{eq:MaxDistanceCenter} is in turn equivalent to
 \begin{align*}
 	c_1(a,\|x\|)\|y_1\|^2+c_2(a,\|x\|)\|y_2\|^2= 	c_1(a,\|x\|)\|y_2\|^2+c_2(a,\|x\|)\|y_1\|^2,
 \end{align*}
which leads to 
\begin{equation}\label{eq:Equation1}
		c_1(a,\|x\|)= c_2(a,\|x\|)
\end{equation}
for $ \|y_1\|\neq \|y_2\|$. This equality has a unique non-negative solution $a(\|x\|)$ given by \eqref{eq_12_06_2}. In particular, in this case we have
\begin{equation}\label{eq:Equation2}
c_1(a(\|x\|),\|x\|)={1\over 2}c_1(a(\|x\|),\|x\|)+{1\over 2}c_2(a(\|x\|),\|x\|)=a(\|x\|)(1+\|x\|^2),
\end{equation}
 and \eqref{eq:CosDistance} becomes
\begin{align*}
	\cos \big(d_s(s_1(L'),z(L))\big) &>  {c_1(a(\|x\|),\|x\|)\|y\|^2+c_2(a(\|x\|),\|x\|)\over (1+\|y\|^2)(1+a(\|x\|)^2\|x\|^2)}\\
	&=	{c_1(a(\|x\|),\|x\|)\over (1+a(\|x\|)^2\|x\|^2)}={a(\|x\|)(1+\|x\|^2)\over 1+a(\|x\|)^2\|x\|^2}.
\end{align*}
Since the function $z\mapsto\arccos(z)$ is decreasing on $(-1,1)$ we conclude that
$$
d_{s}(s_1(L'),z (L))< \beta(\|x\|)
$$
holds for any $L'\in\hat{A}_L$ with equality if and only if $d_h(x,y)=1$. In the exceptional case $x=o$ inequality \eqref{eq_12_06_1} becomes $(1+\|y\|^2)\leq \cosh(1)(1-\|y\|^2)$ and we get
$$
d_s(s_1(L'),z(L))=\arccos\big(\sin\gamma(\|y\|)\big)\leq \arccos(1/\cosh(1))=\beta(0),
$$
with equality if and only if $d_h(0,y)=1$. Thus, 
\[
\{s_1(L')\colon L'\in \hat{A}_L\}\subset \scap(z(L),\beta(\|x\|)).
\]  
It remains to ensure that for any $v\in \scap(z(L),\beta(\|x\|))$ 
there exists $L'\in \hat{A}_L$ such that $s_1(L')=v$. As any line $L'\in \hat A$ 
is uniquely determined by the point $y':=L'\cap H$ it is enough to construct this point.

Given $v=(v_1,\ldots,v_d)\in\BS^{d-1}$ satisfying 
$\langle v,u_{H}\rangle\ge 0$ denote by $v'$ the orthogonal projection 
of $v$ onto $H$. Then there exists a unique vector $u'=v'\|v'\|\in H\cap \BS^{d-1}$ and $\gamma={\pi\over 2}-\arccos\langle v,u_H\rangle\in [0,{\pi\over 2}]$, such that $v=\cos(\gamma)u'+\sin(\gamma)u_H$. Set $y':=(1-\sin(\gamma))\cos(\gamma)^{-1}u'\in H$. Then
\[
\cos(\gamma)={2\|y'\|\over 1+\|y'\|^2}\qquad\text{ and }\qquad \sin(\gamma)={1-\|y'\|^2\over 1+\|y'\|^2}.
\]
Further since $v\in \scap(z(L),\beta(\|x\|))$ and according to \eqref{eq_center_cap} we have
\begin{align*}
\cos(d_s(v,z(L)))&=\cos(\gamma(a(\|x\|)\|x\|))\cos(\gamma)\langle u_x, u'\rangle+\sin(\gamma(a(\|x\|)\|x\|))\sin(\gamma)\\
&={4a(\|x\|)\over (1+a(\|x\|)^2\|x\|^2)(1+\|y'\|^2)}\langle x, y'\rangle+{(1-a(\|x\|)^2\|x\|^2)(1-\|y'\|^2)\over (1+a(\|x\|)^2\|x\|^2)(1+\|y'\|^2)}\\
&\ge {a(\|x\|)(1+\|x\|^2)\over 1+a(\|x\|)^2\|x\|^2}={c_1(a(\|x\|),\|x\|)\|y'\|^2+c_2(a(\|x\|),\|x\|)\over (1+\|y'\|^2)(1+a(\|x\|)^2\|x\|^2)},
\end{align*}
since $a(\|x\|)$ solves \eqref{eq:Equation1} and due to \eqref{eq:Equation2}. The latter implies that 
\begin{align*}
2\langle x, y'\rangle&\ge {c_1(a(\|x\|),\|x\|)\|y'\|^2+c_2(a(\|x\|),\|x\|)\over 2a(\|x\|)}-{(1-a(\|x\|)^2\|x\|^2)(1-\|y'\|^2)\over 2a(\|x\|)}\\
&=\|y'\|^2+\|x\|^2-c(1-\|x\|^2)(1-\|y'\|^2),
\end{align*}
which according to \eqref{eq:Equation3} is equivalent to $d_h(x,y')\leq 1$.
\end{proof}
Using Lemma \ref{lm:scaps} we may naturally identify the Poisson  process $\Psi(\lambda\mu^h)$ of hyperbolic lines with the process of spherical caps in the following way. Consider a measurable mapping 
\begin{equation}\label{eq:F-}
\begin{aligned}
F^{-}:A(d,1)&\mapsto \BS^{d-1}\times (0,\pi],\\
L=L(H,x)&\mapsto \big(z(L), \beta(\Vert x \Vert)\big).
\end{aligned}
\end{equation}
By the mapping theorem \cite[Theorem 5.1]{LP18} for Poisson 
point processes,
$\Psi(\lambda \mu^{s,-}):=F^-(\Psi(\lambda\mu^h))$ is a 
Poisson  process on $\BS^{d-1}\times (0,\pi]$ with intensity 
measure $\lambda \mu^{s,-}:= F^-(\lambda\mu^h)$, the image measure 
of $\lambda\mu^h$ under the mapping $F^-$. As usual, we let
$$
\CC(\lambda \mu^{s,-})
=\bigcup_{(u,\alpha)\in\Psi(\lambda \mu^{s,-})}\scap(u,\alpha)
$$
be the random subset of $\BS^{d-1}$ covered by the caps induced by 
the point process $\Psi(\lambda \mu^{s,-})$.
\begin{remark} \label{remark:uH}
We make two observations that may be of use. Firstly, the center point 
$z(L)$ is not equal to $s_1(L)$ as seen by comparing \eqref{eq_center_cap} to \eqref{eq_s}.
Secondly, if $\pm u_H$ were not both represented in \eqref{eq_plane_intensity}, 
the resulting cap process $\Psi(\lambda \mu^{s,-})$
would not be rotationally invariant. For instance, if $u_H$ would always have positive 
first coordinate, the resulting cap process would be very different from the current 
one. 
\end{remark}

\begin{lemma}\label{lm:coupling}
There is a coupling of $\Psi(\lambda \mu^{s,-})$ 
and $\Psi(\lambda\mu^h)$ such that the following holds: if $\CC(\lambda \mu^{s,-})$ is connected, then $\CC(\lambda\mu^h)$ is connected.
\end{lemma}
\begin{proof}
Consider two arbitrary hyperbolic lines $L=L(H,x)$ and $L'=L(H',x')$. We will show that if the open caps $S:=\scap(z(L),\beta(\|x\|))$ and $S':=\scap(z(L'),\beta(\|x'\|))$ intersect, then the corresponding cylinders $\cc_L$ and $\cc_{L'}$ intersect as well. Let $u\in S\cap S'$, then there is some $\varepsilon>0$ such that $\scap(u,\varepsilon)\subset S\cap S'$. Let $F\subset \BH^d$ be the totally geodesic $(d-1)$-dimensional subspace, such that 
$\Pi_{\rm rad}(F)=\scap(u,\varepsilon)$.
 By construction of $S$ and $S'$ from Lemma \ref{lm:scaps} it is clear that for any point $z\in F$ we can find geodesics $M\in \cc_L$ and $M'\in \cc_{L'}$ passing through $z$. In particular, this means that $s_1(M),s_1(M')\in \scap(u,\varepsilon)$. Thus, 
$\emptyset \neq  F \subset \cc_L\cap \cc_{L'}$.
\end{proof}

There is another "larger" spherical cap process, associated with $\Psi(\lambda\,\mu^h)$ which we are going to use later in this paper. Namely, given some $r\ge r_0:=\sqrt{(\cosh(1)-1)/(\cosh(1)+1)}$ let us define the set
$$
A^r:=\big\{L(H,x)\in A(d,1)\colon \|x\|>r\big\},
$$
which consists of all lines $L(H,x)\in A(d,1)$ with $d_h(o,x)> 1$. In particular the latter means that $o\not\in \cc_L$ for any $L\in A^r$. Further consider the mapping
\begin{equation}\label{eq:F+}
\begin{aligned}
F^+_r: A^r &\to \BS^{d-1}\times  (0,\pi],\\
L(H,x)&\mapsto \big(\Pi_{\rm rad}(x), \gamma(\|x\|)+2\beta(\Vert x \Vert)\big).
\end{aligned}
\end{equation}
It should be noted, that for any hyperbolic line $L(H,x)\in A^r$ we have 
$$
\gamma(\|x\|)+2\beta(\|x\|)<\pi
$$
by \eqref{eq_alpha} and since the function $\beta(r)$ is strictly decreasing with $\beta(r_0)<{\pi\over 4}$. 
Hence, the mapping $F^+_r$ is well-defined for any $r\ge r_0$. Again by the mapping property and the restriction property of Poisson  processes (see \cite[Theorem 5.1]{LP18} and \cite[Theorem 5.2]{LP18}, respectively), we conclude that 
$\Psi(\lambda \mu^{s,+}_{r}):=F^+_r(\Psi(\lambda\mu^h)\cap A^r)$ is a Poisson  process on $\BS^{d-1}\times  (0,\pi]$ with intensity 
measure $\lambda \mu^{s,+}_{r}:= F^+_r(\lambda\mu^h|_{A^r})$, where $\mu^h|_{A^r}$ denotes the restriction of the measure $\mu^h$ to the set $A^r$. Note also that by Lemma \ref{lm:scaps} for any $\cc_{L}$ with $L\in A^r$ the radial projection $\Pi_{\rm rad}(\cc_L)$ of the cylinder $\cc_L$ is well defined, since $o\notin \cc_L$ and by Lemma \ref{lm:scaps} we have
\begin{equation} \label{eqn:PicLsubset}
\Pi_{\rm rad}(\cc_L) \subset 
\scap\big(\Pi_{\rm rad}(x), \gamma(\|x\|)+2\beta(\Vert x \Vert)\big).
\end{equation}
\begin{remark}
We observe that $\gamma(\|x\|)+2\beta(\Vert x \Vert)$ is chosen 
as the radius of the spherical cap in the right hand side of \eqref{eqn:PicLsubset}
since $\gamma(\|x\|)$ is half the length of the radial projection of $L(H,x),$ while
$2\beta(\Vert x \Vert)$ is the diameter of the ``endcap'' 
$\scap(z(L),\beta(\Vert x \Vert)).$ The reason for using the diameter 
$2\beta(\Vert x \Vert)$ rather than the radius $\beta(\Vert x \Vert)$ is 
the fact 
that $s_1(L)$ is not the center of the cap $\scap(z(L),\beta(\Vert x \Vert))$
 (as observed in Remark \ref{remark:uH}). Of course, the inclusion \eqref{eqn:PicLsubset} may seem to be somewhat 
``wasteful'' since
the radial projection of $\cc_L$ will be much smaller than 
$\scap\big(\Pi_{\rm rad}(x), \gamma(\|x\|)+2\beta(\Vert x \Vert)\big),$ but it is 
enough for our purposes.
\end{remark}
From the construction it is evident that there is a coupling of the two Poisson  processes
$\Psi(\lambda \mu^{s,+}_{r})$ 
and $\Psi(\lambda\mu^h)$ such that 
\begin{equation} \label{eqn:couplingcc+}
\Pi_{\rm rad}\Big( \bigcup_{L\in\Psi(\lambda \mu^h)\cap A^r}\cc_L\Big)
\subset
\CC(\lambda \mu^{s,+}_{r}).
\end{equation}
Concerning the random sets $\CC(\lambda \mu^{s,-})\subset\BS^{d-1}$ 
and $\CC(\lambda \mu^{s,+}_{r})\subset\BS^{d-1}$ 
of caps induced by the Poisson 
point processes $\Psi(\lambda \mu^{s,-})$ and 
$\Psi(\lambda \mu^{s,+}_{r})$, respectively, we prove 
the following two results from which we will derive our main 
theorem below.

\begin{theorem} \label{thm:cover}
There exists a constant $\lambda_1<\infty$ only depending on $d$, 
such that for all $\lambda>\lambda_1$ we have that
\[
\BP(\CC(\lambda \mu^{s,-})=\BS^{d-1})=1.
\]
\end{theorem}

\begin{theorem} \label{thm:disconnect}
There exists a constant $\lambda_2>0$ only depending on $d$, such that for all $0<\lambda<\lambda_2$ and $r>{1\over 2}\big(\tan\big({e\pi\over e^2-e+1}\big)\big)^{-1}$ we have that
\[
\BP(\CC(\lambda \mu^{s,+}_{r}) \textrm{ is not connected})>0.
\]
\end{theorem}

In order to complete the proof of Theorem
 \ref{thm:ConnectivityPhaseTransition} we need to prove the 
 monotonicity in 
 $ \lambda  $ of the probability for $\CC(\lambda \mu^h) $ being connected. This is formulated in the subsequent lemma.
\begin{lemma}
Let $ \lambda $ be such that $\BP( \CC(\lambda \mu^h) \text{ is connected})=1 $. Then,
\[
\BP( \CC(\lambda' \mu^h) \text{ is connected})=1 
\]
 for every $\lambda'> \lambda $.
\end{lemma}
\begin{proof}
Let $\tilde{\lambda}=\lambda'-\lambda$ and consider the Poisson  processes $\Psi(\tilde{\lambda}\mu^h)$ and $ \Psi(\lambda\mu^h) $
which we assume are independent, and defined on the same probability 
space. By the superposition property of Poisson  processes (see \cite[Theorem 3.3]{LP18}), it follows
that $ \Psi(\lambda\mu^h) \cup \Psi(\tilde{\lambda}\mu^h)$ has 
the same distribution as $\Psi(\lambda'\mu^h)$. Since the number of 
hyperbolic lines in $\Psi(\tilde{\lambda}\mu^h)$ is a.s.\ countable, the 
statement follows from the union bound if we can show that for any fixed line 
$L\in A_h(d,1),$ we have that 
$\BP(L\cap \CC(\lambda \mu^h)\neq \emptyset)=1.$

To this end, fix some $ L=L(H,x) \in A_h(d,1)$.  By Lemma \ref{lm:scaps},
the cylinder $\cc_L$ induces a cap 
$ \operatorname{cap}(z(L), \beta(\|x\|))\subset\BS^{d-1}$ 
which has positive $\sigma_{d-1}$-measure, see also Lemma
\ref{lem:GammaBetaComparison} below. We let $ \operatorname{CS}(z(L), \beta(\|x\|))$ be
the sector of $ \mathbb{B}^d $ whose boundary is given 
by $ \operatorname{cap}(z(L),\beta(\|x\|)) $, namely
$\operatorname{CS}(z(L), \beta(\|x\|))
=\Pi_{\rm rad}^{-1}\left(\operatorname{cap}(z(L),\beta(\|x\|))\right)\cup\{o\}.$
According to \eqref{eq_transform}, we have that for every 
$ L'=L'(H',y)\in A_{h}(d,1)$ with 
$ y \in  \operatorname{CS}(z(L), \beta(\|x\|))$  and 
$ d_h(o, y) > \operatorname{arcsinh}\big( \cot(\frac{\beta(\|x\|)}{2})\big):= \delta(\|x\|)$, one of the caps induced by $ L'  $ 
intersects $ \operatorname{cap}(z(L), \beta(\|x\|))  $. For 
the associated cylinders $ \cc_L $ and $ \cc_{L'} $ it then holds 
that $ \cc_L \cap \cc_{L'}\neq \emptyset $. With this in consideration, 
we can now conclude that for every 
$ L=L(H,x) \in A_h(d,1) $ it holds that 
	\begin{align*}
	&\BP(L(H,x) \cap   \CC(\lambda \mu^h)= \emptyset ) \\
	&\quad \leq \BP(\vert\{L'(H',y)\in \Psi(\lambda\mu^h): y \in \operatorname{CS}(z(L), \beta(\|x\|))\backslash B^d(o,\delta(\|x\|))\} \vert=0)\\
	&\quad =\exp\big(-\lambda \mu^h(L'(H',y): y \in \operatorname{CS}(z(L), \beta(\|x\|))\backslash B^d(o,\delta(\|x\|))\big)\\
	&\quad =0,
	\end{align*}
where we used the easily verified fact that 
\[
\mu^h(L'(H',y): y \in \operatorname{CS}(z(L), \beta(\|x\|))\backslash B^d(o,\delta(\|x\|)))=\infty.
\]
This proves the claim.
\end{proof}
\begin{proof}[Proof of Theorem \ref{thm:ConnectivityPhaseTransition}
(using Theorem \ref{thm:cover} and Theorem \ref{thm:disconnect})]
First, we consider the connectivity regime. Let  $\cc_{L_a}$ and 
$\cc_{L_b}$ with $L_a=L(H_a,x_a)$, $L_b=L(H_b,x_b)$ be any pair of 
(fixed) cylinders. Consider also the coupling of 
$\Psi(\lambda \mu^{s,-})$ 
and $\Psi(\lambda\mu^h)$ from Lemma \ref{lm:coupling}.
According to Theorem \ref{thm:cover} we have that for all 
$\lambda<\infty$ large enough,
$\CC(\lambda \mu^{s,-})=\BS^{d-1}$ with probability one. For such 
$\lambda$, it follows that there exists an integer $n\in\mathbb{N}$ and a sequence of ``endcaps'' $\scap(s_1(L_k),\beta(\|x_k\|))$, $0\leq k\leq n+1$,
such that 
\[
\scap(s_1(L_k),\beta(\|x_k\|))\cap \scap(s_1(L_{k+1}),\beta(\|x_{k+1}\|)) \neq \emptyset,
\]
for $0\leq k\leq n$ and $L_0=L_a$, $L_{n+1}=L_b$, $\scap(s_1(L_k),\beta(\|x_k\|))\in \Psi(\lambda \mu^{s,-})$ for $1\leq k\leq n$. Clearly, by Lemma \ref{lm:coupling} it follows that 
also the  corresponding cylinders $\cc_{L_a}, \cc_{L_1},\ldots,\cc_{L_n},\cc_{L_b}$ with $\cc_{L_k}\in \Psi(\lambda\mu^h)$ for $1\leq k\leq n$ form a 
connected sequence, and so $\cc_{L_a}$ and $\cc_{L_b}$ belong to the 
same connected component. 

From this we claim that it follows that $\CC(\lambda \mu^h)$ 
is connected with probability $1$. 
To see this, let 
$\cc_L \stackrel{\Psi(\lambda\mu^h)}{\longleftrightarrow} \cc_{L'}$
denote the event that $\cc_{L}$ and $\cc_{L'}$ are connected by using
other cylinders associated to the lines in $\Psi(\lambda\mu^h).$
We let $\mathcal{E}$ be the event that 
$\cc_L \stackrel{\Psi(\lambda\mu^h)}{\longleftrightarrow} \cc_{L'}$
for every $L,L'\in \Psi(\lambda\mu^h),$ i.e. that 
$\CC(\lambda \mu^h)$ is connected. Then, we can use 
the multivariate Mecke formula \cite[Theorem 4.4]{LP18} for 
Poisson  processes to see that
the probability for the complementary event $\mathcal{E}^c$ is bounded by
\begin{align*}
\mathbb{P}(\mathcal{E}^c)& \leq {1\over 2}
\BE\left[\sum_{(L,L')\in\Psi^2(\lambda\mu^h)\atop L\neq L'}
{\bf 1}\left(\left\{
\cc_L \stackrel{\Psi(\lambda\mu^h)}{\longleftrightarrow} \cc_{L'}
\right\}^c\right)
\right]\\
&={\lambda^2\over 2}\int_{A_h(d,1)}\int_{A_h(d,1)}
\BP\left(\left\{
\cc_L \stackrel{\Psi(\lambda\mu^h)\cup\{L,L'\}}{\longleftrightarrow} \cc_{L'}
\right\}^c\right)
\mu^h(\dint L)\mu^h(\dint L').
\end{align*}
However, from the above considerations it follows that the probability under the integral signs is zero for all $L,L'\in A_h(d,1)$, implying that $\mathbb{P}(\mathcal{E}^c)=0$ and, finally, $\mathbb{P}(\mathcal{E})=1$. As a result, $\CC(\lambda \mu^h)$ is connected with probability $1$.

We now prove that for $\lambda>0$ small enough, 
$\CC(\lambda\mu^h)$ is not connected with positive probability.
To that end we start by noting that due to the coupling \eqref{eqn:couplingcc+} connectedness of
$$
\CC_r(\lambda\mu^h):=\bigcup_{L\in\Psi(\lambda \mu^h)\cap A^r}\cc_L
$$
implies connectedness of $\CC(\lambda \mu^{s,+}_{r})$. Thus,
\begin{align*}
&\BP(\CC(\lambda\mu^h)\textrm{ is not connected})\\
&\ge\BP(\CC_r(\lambda\mu^h)\textrm{ is not connected}, \Psi(\lambda \mu^h)\cap (A_h(d,1)\setminus A^r)=\emptyset)\\
&= \BP(\CC(\lambda \mu^{s,+}_{r})\textrm{ is not connected})\,\BP(\Psi(\lambda \mu^h)\cap (A_h(d,1)\setminus A^r)=\emptyset),
\end{align*}
where in the last step we used the independence property of Poisson  processes. Since
\begin{align*}
\BP(\Psi(\lambda \mu^h)\cap (A_h(d,1)\setminus A^r)=\emptyset)&=\exp\big(-\lambda \mu^h(\{L(H,x)\in A_h(d,1)\colon \|x\|\leq r\})\big)\\
&=\exp\Big(-{\lambda\omega_{d-2} \over d-1}\Big({2r\over 1-r^2}\Big)^{d-1}\Big )>0,
\end{align*}
for any $r\in[0,1)$, taking some $r>{1\over 2}\big(\tan\big({e\pi\over e^2-e+1}\big)\big)^{-1}$ and applying Theorem \ref{thm:disconnect} we conclude that
$$
\BP(\CC(\lambda\mu^h)\text{ is not connected})>0,
$$
for all $\lambda>0$ small enough.
\end{proof}

\begin{remark} As stated in the introduction (see Remark \ref{rem:1}), the result 
of part (i) of Theorem \ref{thm:ConnectivityPhaseTransition} is the
best possible. The reason is that even for very low intensities, there
is positive probability that there can be a few lines 
$L\in\Psi(\lambda \mu^h)$ having a `small' distance to the origin 
$o$. In such a situation, the spherical caps 
${\rm cap}(z(L),\beta(\| x\|))\subset\mathbb{S}^{d-1}$ 
associated with the cylinders corresponding to these lines 
(compare with Lemma \ref{lm:scaps}) cover the entire sphere 
$\BS^{d-1}$ and the union set $\CC(\lambda\mu^h)$ 
is necessarily connected. More precisely, let us fix some $R\in(0,\pi)$ and consider a covering of the sphere $\mathbb{S}^{d-1}$ by spherical caps of radius $R$. It is easy to show that the sphere can be completely covered by a finite number $M(R)$ of caps $\scap_i$, $1\leq i\leq M(R)$. Further for any spherical cap $\scap(u,2R)$ with $u\in\scap_i$ we have $\scap_i\subset \scap(u,2R)$. Thus, by Lemma \ref{lm:coupling} we see that $\CC(\lambda\mu^h)$ is connected if $\CC(\lambda \mu^{s,-})=\mathbb{S}^{d-1}$, which in turn happens if $\Psi(\lambda \mu^{s,-})\cap (\scap_i\times [2R,\pi))\neq \emptyset$ for all $1\leq i\leq M(R)$. Due to Lemma \ref{lemma:measure} below, the latter happens with positive probability for any $\lambda>0$.
On the other hand, it is only if 
$\lambda<\lambda_c$ that the union set can be disconnected with 
positive probability.

\end{remark}

\section{Boolean model of spherical caps} \label{sec:Boolcaps}

The proofs of Theorems \ref{thm:cover} and \ref{thm:disconnect} will proceed by 
analyzing the Poisson  processes $\Psi(\lambda \mu^{s,+})$ and 
$\Psi(\lambda \mu^{s,-})$ of Section \ref{sec:inducedBool}. However, 
instead of analyzing these processes directly, we will introduce an additional family of Poisson  processes, giving rise to processes of spherical caps, having slightly larger or smaller spherical radii than the caps produced by $\Psi(\lambda \mu^{s,+})$ and 
$\Psi(\lambda \mu^{s,-})$, respectively. The latter in combination with coupling arguments will allow us to treat $\Psi(\lambda \mu^{s,+})$ and 
$\Psi(\lambda \mu^{s,-})$ simultaneously in a unified way, and this will simplify the arguments.
The key to making this rigorous will lie in understanding the difference between the radius 
$\beta(\Vert x \Vert )$ and the length of the projection $\gamma(\|x\|).$ 
Therefore, we start with the following lemma.

\begin{lemma}\label{lem:GammaBetaComparison}
For any $x\in\BB^d$ with  $\|x\|\in [0,1)$ we have
$$
{2\over \pi}\arccos\big(\cosh^{-1}(1)\big)\gamma(\|x\|)\leq \beta(\|x\|)\leq \sinh(1)\gamma(\|x\|).
$$
\end{lemma}
\begin{proof}
If we use the change of variables $h:=d_h(o,x)$, then we get from 
\eqref{eq:EuclideanHyperbolicDist} that
$$
\|x\|=\sqrt{{\cosh(h)-1\over \cosh(h)+1}},\qquad h\in[0,\infty).
$$
Further we note that by \eqref{eq_transform},
$$
\gamma(\|x\|)=\widetilde\gamma(h):=\arccot\big(\sinh(h)\big).
$$ 
Since $\beta(\Vert x \Vert)$ is the radius of the spherical cap $\scap(z(L),\beta(\Vert x \Vert))$
where $L=L(H,x),$ we have (using Lemma \ref{lm:scaps}) that for any 
$u\in \scap(z(L),\beta(\Vert x \Vert))$ there exists $y\in H$ with $d_h(y,x)<1$ and 
$s_1(L(H,y))=u$. Therefore, by using the representation \eqref{eq_s}, we obtain
\begin{align*}
2\beta(\|x\|)&=\sup_{u_1,u_2\in \scap(z(L),\beta(\Vert x \Vert))}d_s(u_1,u_2)\\
&=\sup_{\substack{y_1,y_2\in H\colon\\ d_h(y_i,x)<1, i=1,2}}\arccos(\cos\gamma(\|y_1\|)\cos\gamma(\|y_2\|)\langle u_{y_1},u_{y_2}\rangle +\sin\gamma(\|y_1\|)\sin\gamma(\|y_2\|)),
\end{align*}
where we recall that $u_y=y/\|y\|$. Using the geometrical interpretation of $\beta(\|x\|)$ as the radius of the spherical cap $ \{s_1(L')\colon L'\in \hat{A}_L\}$ we conclude that the points $y_1,y_2$, for which the supremum in the above expression is achieved, must satisfy $d_h(y_1,x)=d_h(y_2,x)=1$. This is equivalent to the fact that the corresponding hyperbolic lines $L_1=L(H,y_1)$ and $L_2=L(H,y_2)$ belong to the boundary of the closure of cylinder $\cc_L$ and, hence, $s_1(L_1)$, $s_1(L_2)$ belong to the boundary of the closure of the spherical cap $\scap(z(L),\beta(\|x\|))$. At the same time among all points $y_1,y_2\in H$ with $d_h(y_1,x)=d_h(y_2,x)=1$ the supremum in the above expression is achieved whenever $s_1(L_1)$ and $s_1(L_2)$ lie on diameter of a spherical cap $\scap(z(L),\beta(\|x\|))$, which in particular holds if $\langle u_{y_1},u_{y_2}\rangle =1$, in case when $h=d_h(0,x)$ satisfies $h>1$, and if $\langle u_{y_1},u_{y_2}\rangle =-1$ for $h\in [0,1]$. Combining these observations leads to
$$
\beta(\|x\|)={1\over 2}\sup_{\substack{y_1,y_2\in H\colon\\ d_h(y_i,x)=1, i=1,2\\
\langle u_{y_1},u_{y_2}\rangle=1}}\arccos(\cos(\gamma(\|y_2\|)-\gamma(\|y_1\|)))={1\over 2}\big(\widetilde\gamma(h-1)-\widetilde\gamma(h+1)\big)=\widetilde\beta(h),
$$
for $h>1$ and 
\begin{align*}
\beta(\|x\|)& ={1\over 2}\sup_{\substack{y_1,y_2\in H\colon\\ d_h(y_i,x)=1, i=1,2\\
\langle u_{y_1},u_{y_2}\rangle=-1}}\arccos(-\cos(\gamma(\|y_2\|)+\gamma(\|y_1\|)))\\
& ={1\over 2}\big(\pi-\widetilde\gamma(h+1)-\widetilde\gamma(1-h)\big)=\widetilde\beta(h),
\end{align*}
for $h\in[0,1]$. It is easy to check that the function $h\mapsto$
$\widetilde\gamma(h)/\widetilde\gamma(h+1)$ is increasing for $h>0$. As a consequence,
$\widetilde\beta(h)/\widetilde\gamma(h)$ is increasing for $h>1$ as well. In the same 
way we see that the function $\widetilde\beta(h)/\widetilde\gamma(h)$ is 
also increasing for $h\in[0,1]$.
Thus, since we set $\gamma(0)={\pi\over 2}$ and $\beta(0)=\arccos(\cosh^{-1}(1))$ we have
$$
{2\over \pi}\arccos\big(1/\cosh(1)\big)={\beta(0)\over \gamma(0)}\leq {\beta(\|x\|)\over \gamma(\|x\|)}={\widetilde\beta(h)\over \widetilde\gamma(h)}\leq \lim_{h\to\infty}{\widetilde\beta(h)\over \widetilde\gamma(h)},
$$
if $\|x\|\in[0,1)$. In order to compute the limit we will use l'Hospital rule. Using \eqref{eq_transform} we get
\begin{align*}
\lim_{h\to\infty}{\widetilde\beta(h)\over \widetilde\gamma(h)}&={1\over 2}\lim_{h\to\infty}{\arccot\big(\sinh(h-1)\big)-\arccot\big(\sinh(h+1)\big)\over \arccot\big(\sinh(h)\big)}\\&={1\over 2}\lim_{h\to\infty}{-\cosh(h-1)^{-1}+\cosh(h+1)^{-1}\over -\cosh(h)^{-1}}\\
&=\sinh(1)\lim_{h\to\infty}{\sinh(h)\cosh(h)\over \cosh(h-1)\cosh(h+1)}=\sinh(1).
\end{align*}
This completes the argument.
\end{proof}

The proofs of both Theorem \ref{thm:cover} and Theorem \ref{thm:disconnect} 
will utilize Lemma \ref{lem:GammaBetaComparison} in order to 
compare $\Psi(\lambda \mu^{s,+})$ and $\Psi(\lambda \mu^{s,-})$ to a
spherical cap processes, where the radii 
will be given by $\big\{c\gamma(\|x\|)\colon L(H,x)\in \Psi(\lambda\mu_h)\big\}$. More precisely for a given constant $c>0$ define the set
$$
A_c:=\left\{L(H,x)\in A_h(d,1)\colon \gamma(\|x\|)< {\pi\over \max(c,2)}\right\}.
$$
This set consists of all lines $L(H,x)\in A_h(d,1)$, such that the set $\scap(\Pi_{\rm rad}(x),c\gamma(\|x\|))$ is a spherical cap, namely it has spherical radius at most $\pi$ and, hence, does not ``wrap around'' the 
entire sphere. 
Further consider the measurable mapping 
\begin{equation}\label{eq:Ftilde}
\begin{aligned}
\tilde{F}_{c,f}: A_c &\to \BS^{d-1}\times  (0,\pi],\\
L(H,x)&\mapsto \Big(\cos\big(f(\|x\|)\big)\Pi_{\rm rad}(x)+\sin\big(f(\|x\|)\big)u_{H}, c\gamma(\|x\|)\Big).
\end{aligned}
\end{equation}
where $f:[0,1]\mapsto [0,\pi/2]$ is some measurable map. By the mapping property of Poisson  processes \cite[Theorem 5.1]{LP18} we conclude that 
$\Psi(\lambda\mu_{c,f}^s):=\tilde{F}_{c,f}(\Psi(\lambda \mu^h)\cap A_c)$ is a 
Poisson  process on $\BS^{d-1}\times (0,\pi]$ with intensity measure 
\begin{equation}\label{eq:MuFC}
\lambda\mu^{s}_{c,f}:= \tilde{F}_{c,f}(\lambda\mu^h|_{A_c}).
\end{equation}

For future reference we note that by taking $f(\|x\|)=f_1(\|x\|):=\gamma(a(\|x\|)\|x\|)$ where $a(\|x\|)$ is the function of $\|x\|$ defined in \eqref{eq_12_06_2} and taking $c={1\over 2}\arccos(\cosh^{-1}(1))$ from the lower bound in Lemma \ref{lem:GammaBetaComparison}, there is a natural coupling of the processes $\Psi(\lambda \mu^{s,-})$ and $\Psi(\lambda\mu_{c,f_1}^s)$ such that almost surely
\begin{equation} \label{eqn:C-inclusion}
\CC(\lambda\mu^{s}_{c,f_1})\subset\CC(\lambda \mu^{s,-}).
\end{equation}
Indeed, this is possible since both processes may be constructed as an image of the same realization of $\Psi(\lambda\mu_h)$ under mappings $F^{-}$ in \eqref{eq:F-} and $\tilde F_{c,f_1}$ in \eqref{eq:Ftilde}, respectively. Observe also that by \eqref{eq_center_cap}, our 
choice of the function $f_1$ and the constant $c$, the right hand side of \eqref{eq:Ftilde} becomes 
$(z(L),{1\over 2}\arccos(\cosh^{-1}(1))\gamma(\|x\|)).$
Then, inclusion holds due to Lemma \ref{lem:GammaBetaComparison}.
On the other hand taking $f(\|x\|)=f_2(\|x\|)=0$ and $c_0=2\sinh(1)+1=e-e^{-1}+1$, by Lemma \ref{lem:GammaBetaComparison} we obtain a natural coupling of the processes $\Psi(\lambda\mu_{r}^{s,+})$ and 
$\Psi(\lambda\mu_{c_0,f_2}^s)$ such that almost surely
\begin{equation}\label{eqn:C+inclusion}
\CC(\lambda \mu^{s,+}_{r})\subset\CC(\lambda\mu^{s}_{c_0,f_2}),
\end{equation}
for any $r>r_1$, ensuring that $A^r\subset A_{c_0}$. In particular, 
the latter holds for
$
r_1={1\over 2}\big(\tan\big({e\pi\over e^2-e+1}\big)\big)^{-1}.
$
Again the coupling is possible since with this choice of $r$ and $c_0$ both processes may be constructed as an image of the same realization of $\Psi(\lambda\mu_h)$ under mappings $F^{+}_r$ in \eqref{eq:F+} and $\tilde F_{c_0,f_2}$ in \eqref{eq:Ftilde}, respectively, and according to Lemma \ref{lem:GammaBetaComparison} we have $2\beta(\|x\|)+\gamma(\|x\|)\leq 
c_0\gamma(\|x\|)$. 

In the next lemma we determine the intensity measure 
$\lambda\mu^{s}_{c,f}$ of $\Psi(\lambda\mu_{c,f}^s)$, which turns out to be independent of the particular choice of function $f$. On one hand, this is natural to expect, since according to the definition of $\tilde F_{c,f},$ the function $f$ influences the density of the measure $\mu^{s}_{c,f}$ only with respect to the first component, while due to the rotation invariance of $\mu^h$ it is to be expected that $\mu^{s}_{c,f}$ will also be rotationally invariant with respect to the first component. On the other hand, the function $f$ depends additionally on $\|x\|$, which apriori makes the two components of the measure $\mu^{s}_{c,f}$ dependent on each other.

\begin{lemma} \label{lemma:measure}
For any $c>0$ and any measurable function $f:[0,1]\mapsto [0,\pi/2]$ the intensity measure $\lambda\mu^{s}_{c,f}$ of the Poisson 
point process $\Psi(\lambda\mu^{s}_{c,f})$ is of the form
$$
\mu^{s}_{c,f}(A\times B)
=\sigma_{d-1}(A)\int_{0}^{c\pi/\max(2,c)}g_c(\alpha){\bf 1}\{\alpha\in B\}\dint \alpha,
$$
where $A\in\CB(\BS^{d-1})$ and $B\in \CB( (0,\pi])$ are Borel sets and
\begin{equation}\label{eq:gcexact}
g_c(\alpha)=c^{-1}{\cos(c^{-1}\alpha)^{d-2}\over \sin(c^{-1}\alpha)^d}.
\end{equation}
In addition, the density function $g_c$ satisfies
\begin{equation}\label{eq:gcestimate}
c^{d-1}\cos(\alpha_0/c)^{d-2}\alpha^{-d}\leq g_c(\alpha)\leq c^{d-1}\Big({\pi\over 2}\Big)^d\alpha^{-d},
\end{equation}
for any $\alpha\in (0,\alpha_0]$ and $0<\alpha_0<c\pi/\max(2,c)$.
\end{lemma}
\begin{proof}
For Borel sets $A\in\CB(\BS^{d-1})$ and $B\in \CB( (0,\pi])$ we have
\begin{align*}
\mu^{s}_{c,f}(A\times B) &=\mu^h\big(\tilde{F}_{c,f}^{-1}(A\times B)\big)\\
&=\omega_{d}^{-1}\int_{\BS^{d-1}}\int_{u^{\perp}}\cosh(d_h(x,o))\,{\bf 1}\{L(u^{\perp},x)\in\tilde{F}_{c,f}^{-1}(A\times B)\}\,\CH^{d-1}(\dint x)\sigma_{d-1}(\dint u)
\end{align*}
by definition of $\lambda\mu^{s}_{c,f}(A\times B)$ as an image measure of $\lambda\mu^h$ under the mapping $\tilde F_{c,f}$ defined by \eqref{eq:Ftilde} and the representation \eqref{eq_plane_intensity} of the invariant measure $\mu^h$. Since $H(u):=u^{\perp}\in G_h(d,d-1)$ is isometric to $\BH^{d-1}$ we can introduce hyperbolic polar coordinates in the geodesic hyperplane $H(u)$ and write
\begin{align*}
\mu^{s}_{c,f}(A\times B) &=\omega_{d}^{-1}\int_{\BS^{d-1}}\int_{\BS_{H(u)}^{d-2}}\int_0^{\infty}\cosh(h)\sinh(h)^{d-2}\\
&\qquad\qquad\times{\bf 1}\{L(H(u),t_{h,v}(o))\in\tilde{F}_{c,f}^{-1}(A\times B)\}\,\dint h\,\CH^{d-2}(\dint v)\,\sigma_{d-1}(\dint u),
\end{align*}
where 
$t_{h,v}$ stands for hyperbolic translation operator by distance $h$ in direction $v$, namely it moves $o$ along the geodesic ray in direction $v$ (which is a unit vector in the tangent space of $H(u)$ at $o$) by hyperbolic distance $h$, and $\BS_{H(u)}^{d-2}$ stands for the unit sphere in the tangent space of $H(u)$ at $o$. 

Next, we put
$$
g(h):=f\left(\sqrt{\cosh(h)-1\over \cosh(h)+1}\right).
$$ 
Then according to \eqref{eq:EuclideanHyperbolicDist} and \eqref{eq_transform} we have $g(h)=f(\|x\|)$ and $\gamma(\|x\|)=\arccot(\sinh(h))$. Hence, denoting $\Pi_{\rm rad}(t_{h,v}(o))=\tilde v$ and by applying Fubini's theorem we obtain
\begin{align*}
\mu^{s}_{c,f}(A\times B)&=\omega_{d}^{-1}\int_0^{\infty}\cosh(h)\sinh(h)^{d-2}\,{\bf 1}\{c\,\arccot(\sinh(h))\in B\cap(0,c\pi/\max(c,2)]\}\\
&\times\int_{\BS^{d-1}}\int_{\BS_{H(u)}^{d-2}}{\bf 1}\{\cos (g(h))\tilde v+\sin (g(h))u\in A\}\CH^{d-2}(\dint v)\,\sigma_{d-1}(\dint u)\,\dint h.
\end{align*}
For a fixed $B\in \CB( (0,\pi])$ define the following measure on the unit sphere $\BS^{d-1}$:
$$
\Theta(\,\cdot\,):=\mu^{s}_{c,f}(\,\cdot\,\times B).
$$
Let $\rho\in\BS\BO_d$, where $\BS\BO_d$ is the group of rotations in $\BR^d$ and let $A\subset \CB(\BS^{d-1})$.
 Then
\begin{align*}
\Theta(\rho A)&=\omega_{d}^{-1}\int_0^{\infty}\cosh(h)\sinh(h)^{d-2}\,{\bf 1}\{c\,\arccot(\sinh(h))\in B\cap (0,c\pi/\max(c,2)]\}\\
&\qquad\times\int_{\BS^{d-1}}\int_{\BS_{H(u)}^{d-2}}{\bf 1}\{\cos (g(h))\widetilde v+\sin(g(h)) u\in \rho A\}\CH^{d-2}(\dint v)\,\sigma_{d-1}(\dint u)\,\dint h.
\end{align*}
The inner integral can be rewritten as 
\begin{align*}
I(h,\rho A)&:=\int_{\BS^{d-1}}\int_{\BS_{H(u)}^{d-2}}{\bf 1}\{\cos (g(h)) \widetilde v+\sin (g(h))u\in \rho A\}\,\CH^{d-2}(\dint v)\,\sigma_{d-1}(\dint u)\\
&=\int_{\BS^{d-1}}\int_{\BS_{H(u)}^{d-2}}{\bf 1}\{\cos (g(h)) (\rho^{-1}\widetilde v)+\sin (g(h)) (\rho^{-1}u)\in A\}\,\CH^{d-2}(\dint v)\,\sigma_{d-1}(\dint u).
\end{align*}
Note  that for any fixed $H(u)\in G_h(d,d-1)$ the rotation $\rho^{-1}$ can be decomposed as $\rho^{-1} = \rho_1\rho_2$, where $\rho_1,\rho_2\in\BS\BO_d$ are such that $\rho_2H(u)=H(u)$. Hence, $\rho^{-1}u=\rho_1u$. Due to  the invariance of the Hausdorff measure $\CH^{d-2}$ on $\BS_{H(u)}^{d-2}$ with respect to rotations fixing $H(u)$ we have
\begin{align*}
I(h,\rho A)&=\int_{\BS^{d-1}}\int_{\BS_{H(u)}^{d-2}}{\bf 1}\{\cos (g(h))(\rho_1\rho_2\widetilde v)+\sin (g(h))(\rho_1 u)\in A\}\,\CH^{d-2}(\dint v)\,\sigma_{d-1}(\dint u)\\
&=\int_{\BS^{d-1}}\int_{\BS_{H(u)}^{d-2}}{\bf 1}\{\cos (g(h))(\rho_1\widetilde v)+\sin (g(h))(\rho_1 u)\in A\}\,\CH^{d-2}(\dint v)\,\sigma_{d-1}(\dint u).
\end{align*}
Let us define $u':=\rho_1u$, $v'=\rho_1 v$ and observe that 
$\rho_1 H(u)=H(u').$ Then
\begin{align*}
I(h,\rho A)&=\int_{\BS^{d-1}}\int_{\BS_{H(u')}^{d-2}}{\bf 1}\{\cos (g(h)) \widetilde v'+\sin (g(h))u'\in A\}\CH^{d-2}(\dint v')\,\sigma_{d-1}(\dint \rho_1^{-1} u')\\
&=\int_{\BS^{d-1}}\int_{\BS_{H'}^{d-2}}{\bf 1}\{\cos (g(h)) \widetilde v'+\sin (g(h))u'\in A\}\CH^{d-2}(\dint u')\,\sigma_{d-1}(\dint u'),
\end{align*}
since $\sigma_{d-1}$ is $\BS\BO_d$-invariant. We have thus shown that $I(h,\rho A)=I(h,A)$ for any $\rho\in\BS\BO_d$ and, hence,
$$
\Theta(\rho A)=\Theta(A).
$$
This means that the measure $\Theta$ is $\BS\BO_d$-invariant and since any $\BS\BO_d$-invariant measure on $\BS^{d-1}$ is proportional to $\sigma_{d-1}$ according to \cite[Theorem 13.1.3]{SW} we conclude that
$$
\Theta(A)=\mu^{s}_{c,f}(\,\cdot\,\times B)=c(B)\sigma_{d-1}(A).
$$
By taking $A=\BS^{d-1}$ we find
$$
c(B)=\int_0^{\infty}\cosh(h)\sinh(h)^{d-2}\,{\bf 1}\{c\,\arccot(\sinh(h))\in B\cap (0,c\pi/\max(c,2)]\}\,\dint h,
$$
and finally we get
\begin{align*}
\mu^{s}_{c,f}(A\times B)&=\sigma_{d-1}(A)\int_0^{\infty}\cosh(h)\sinh(h)^{d-2}\,{\bf 1}\{c\,\arccot(\sinh(h))\in B\cap (0,c\pi/\max(c,2)]\}\,\dint h.
\end{align*}
Consider a new variable $\alpha:=c\,\arccot(\sinh(h))$. 
Then, 
\begin{align*}
\sinh(h)&=\cot(c^{-1}\alpha), \\
\cosh(h)&=\sqrt{1+\cot(c^{-1}\,\alpha)^2},\\
\dint h&=-{1\over \sqrt{1+\cot(c^{-1}\alpha)^2}}{c^{-1}\over \sin(c^{-1}\alpha)^2}\dint \alpha
%={c^{-1}\over \sqrt{1+\cot(c^{-1}\alpha)^2}}\Big({c^2\over \alpha^2}+{1\over 3} +O(\alpha^2)\Big), 
\end{align*}
and so
\begin{align*}
\mu^{s}_{c,f}(A\times B)&=c^{-1}\,\sigma_{d-1}(A)\int_0^{c\pi/\max(2,c)}{\cot(c^{-1}\alpha)^{d-2}\over \sin(c^{-1}\alpha)^2}\,{\bf 1}\{\alpha\in B\}\,\dint \alpha.
\end{align*}
Thus, 
$$
g_c(\alpha)=c^{-1}{\cot(c^{-1}\alpha)^{d-2}\over \sin(c^{-1}\alpha)^2}
$$
for $\alpha\in(0,c\pi/\max(2,c)]$. Moreover, since $2x/\pi\leq \sin(x)\leq x$, which holds for any $0\leq x\leq {\pi\over 2}$, and since the cosine is a decreasing function on $[0,{\pi\over 2}]$, we have that for every $\alpha\in (0,\alpha_0]$ it holds that
\[
c^{d-1}\Big({\pi\over 2}\Big)^d\alpha^{-d}\ge g_c(\alpha)=c^{-1}{\cos(c^{-1}\alpha)^{d-2}\over \sin(c^{-1}\alpha)^d}
\geq c^{d-1}\cos(\alpha_0/c)^{d-2}\alpha^{-d}.
\]
This completes the argument.
\end{proof}

\section{Proof of Theorem \ref{thm:cover}.}\label{sec:Proof1}

The purpose of this section is to prove that for $\lambda\in(0,\infty)$
large enough, the sphere $\BS^{d-1}$ is covered by a spherical 
cap process $\Psi(\lambda\mu^{s,-})$ related to our original hyperbolic Poisson cylinder model.
The general strategy is based on \cite{Broman}, although the geometry 
is different in our current setting.
Here, we will consider any measure of the form given in Lemma
\ref{lemma:measure}. That is, we let the intensity measure $\lambda \mu^s_{c}$ of 
the Poisson  process $\Psi(\lambda \mu^s_{c})$ be given by
\begin{equation}\label{eq:mu_s}
\mu^s_{c}(A\times B)=\sigma_{d-1}(A)\int_{0}^{c\pi/\max(2,c)}g_c(\alpha){\bf 1}\{\alpha\in B\}\dint \alpha,
\end{equation}
where $c\in(0,\infty)$, $d\ge 2$, and $A\in\CB(\BS^{d-1})$, 
$B\in \CB((0,\pi])$ are Borel sets. Furthermore,  we assume that 
the density $g_c$ satisfies 
the lower bound in \eqref{eq:gcestimate}. 
Clearly, $\Psi(\lambda \mu^s_{c})$
is a Poisson  process on $\BS^{d-1} \times (0,{c\pi\over \max(2,c)}],$ but as before, we 
identify the pair $(x,\alpha)\in \Psi(\lambda \mu^s_{c})$ with the corresponding
spherical cap centred at $x$ and with spherical radius $\alpha.$

As we will see, 
Theorem \ref{thm:cover} is a direct consequence of the following result.

\begin{theorem} \label{thm:cover_gen}
For any measure $\mu^s_{c}$ of the form \eqref{eq:mu_s}, there exists 
some $\lambda_1(d,c)<\infty$ such that for every 
$\lambda> \lambda_1(d,c),$ we have that
\[
\BP(\CC(\lambda \mu^s_{c})=\BS^{d-1})=1.
\]
\end{theorem}
Before we prove the theorem, we introduce some further
notation and prove an auxiliary lemma. For any $n\in \BN$ we define $S_n \subset \BS^{d-1}$ to be a maximal (with 
respect to the number of points) subset of $\BS^{d-1}$ such that 
$d_{s}(x,y) \geq 2^{-n}$ for any $x,y\in S_n$. 
We will assume that $S_n$ is picked according to some predetermined 
rule. The following lemma provides estimates on the cardinality $|S_n|$ of $S_n$. 

\begin{lemma} \label{lemma:cardSnest}
For any $d\ge 2$ there exist constants $0<c(d),C(d)<\infty$ depending only on $d$ 
such that for every $n\geq 1$ we have that
\[
c(d)2^{(d-1)n} \leq |S_n| \leq C(d)2^{(d-1)n}.
\] 
\end{lemma}
\begin{proof}
Take $\varepsilon>0$ small enough and let $S_\varepsilon$ be a 
maximal subsets of $\BS^{d-1}$ satisfying $d_s(x,y)\geq\varepsilon$. 
Then the union $\bigcup_{x\in S_\varepsilon}B\big(x,{\varepsilon\over 2}\big)$ does not touch the ball $B\big(o,1-{\varepsilon\over 2}\big)$, but is contained in $B\big(o,1+{\varepsilon\over 2}\big)$. Write $v(r)$ and $s(r)$ for the volume and the surface area of a ball of radius $r>0$. It follows that
\begin{align*}
|S_\varepsilon|v\Big({\varepsilon\over 2}\Big) &\leq v\Big(1+{\varepsilon\over 2}\Big) - v\Big(1-{\varepsilon\over 2}\Big)
=\int_{1-\varepsilon/2}^{1+\varepsilon/2}s(r)\,dr
\leq \varepsilon s\Big(1+{\varepsilon\over 2}\Big)
=\varepsilon\Big(1+{\varepsilon\over 2}\Big)^{d-1}s(1)
\end{align*}
and hence
\begin{align}\label{eq:14-05-24A}
|S_\varepsilon| \leq \varepsilon\Big(1+{\varepsilon\over 2}\Big)^{d-1}\Big({2\over\varepsilon}\Big)^d{s(1)\over v(1)} = 2d\Big(1+{2\over\varepsilon}\Big)^{d-1}.
\end{align}
On the other hand, it holds that the union $\bigcup_{x\in S_\varepsilon}B\big(x,2\varepsilon\big)$ contains the spherical annulus $B\big(o,1+\varepsilon\big)\setminus B\big(o,1-\varepsilon\big).$ To see this, let for  
$z\in B(o,1+\varepsilon)\setminus B(o,1-\varepsilon),$ $x_z$
be the point in $S_\varepsilon$ which is closest to $z/\Vert z \Vert
\in B(o,1).$
Then,
we have that $d(z,x_z)\leq d(z,z/\Vert z \Vert)+d(z/\Vert z \Vert,x_z)
< 2\epsilon,$ where the last inequality 
follows from the definition of $S_\epsilon.$ We therefore obtain
\begin{align*}
|S_\varepsilon|v(2\varepsilon) &\geq v\left(1+\varepsilon\right) 
- v\left(1-\varepsilon\right)  \\
&  = \int_{1-\varepsilon}^{1+\varepsilon}s(r)\,dr
\geq 2\,\varepsilon s\Big(1-{\varepsilon}\Big)
=2\,\varepsilon\Big(1-{\varepsilon}\Big)^{d-1}s(1).
\end{align*}
It follows that
\begin{align}\label{eq:14-05-24B}
|S_\varepsilon|\geq 2\,\varepsilon\left(1-\varepsilon\right)^{d-1}{1\over(2\varepsilon)^d}{s(1)\over v(1)} = {d\over 2^{d-1}}\left({1\over\varepsilon}-1\right)^{d-1}.
\end{align}
Now, we choose $\varepsilon=2^{-n}$. Then \eqref{eq:14-05-24A} implies that
\begin{align*}
|S_n| \leq 2d(1+2^{n+1})^{d-1} \leq 2d\,2^{(n+2)(d-1)} = d2^{2d-1} \cdot 2^{n(d-1)},
\end{align*}
whereas from \eqref{eq:14-05-24B} we get
\begin{align*}
|S_n| \geq {d\over 2^{d-1}}\left(2^n-1\right)^{d-1} 
\geq {d\over 2^{d-1}}\,2^{(n-1)(d-1)} = {d\over 2^{2(d-1)}}\cdot 2^{n(d-1)}.
\end{align*}
This proves the lemma.
\end{proof}

\noindent
For any $n\in \BN,$ define 
\[
\Psi_n(\lambda \mu^s_{c})
:=\{(x,\alpha)\in \Psi(\lambda \mu^s_{c}): \alpha \geq 2^{-n}\},
\]
which is the collection of pairs $(x,\alpha)$ corresponding to 
caps of spherical radius at least $2^{-n}.$ Next, we let 
\[
\CC_n(\lambda \mu^s_{c}):=\bigcup_{(x,\alpha)\in \Psi_n(\lambda \mu^s_{c})} \scap(x,\alpha)
\qquad\textrm{ and }\qquad
\CC(\lambda \mu^s_{c})=\bigcup_{(x,\alpha)\in \Psi(\lambda \mu^s_{c})} \scap(x,\alpha)
\]
be the associated covered regions and note that $\CC(\lambda \mu^s_{c})=\bigcup_{n\in\BN} \CC_n(\lambda \mu^s_{c}).$

\begin{proof}[Proof of Theorem \ref{thm:cover_gen}]
Throughout the proof $C(d,c)$ denotes a constant, which may depend only on $d$ and $c$ and may change from 
line to line.

Consider any $n\geq N$ where $N$ is some finite number depending only on 
$d$ that we will specify later. Fix some $x_0\in S_n,$ and 
consider the spherical cap $\scap(x_0,2^{-n}).$ Our goal is to compute 
$\lambda \mu^s_{c}\big(\{(x,\alpha): \scap(x_0,2^{-n})\subset \scap(x,\alpha)\}\big)$. 
In order to do this 
we note that for fixed $\alpha>2^{-n}$ we have that 
$\scap(x_0,2^{-n})\subset \scap(x,\alpha)$ if and only if
$d_{s}(x,x_0)\leq \alpha-2^{-n}.$
Recall from Section \ref{sec:Notation} that $\omega_d$ denotes 
the surface area of $\BS^{d-1},$ and 
observe that 
\begin{eqnarray*} 
\lefteqn{\sigma_{d-1}(\{x\in \BS^{d-1}: 
\scap(x_0,2^{-n})\subset \scap(x,\alpha)\})}\\
& & =\sigma_{d-1}(\scap(x_0,\alpha-2^{-n}))
=\omega_{d}\int_0^{\alpha-2^{-n}} \sin^{d-2} \theta\, \dint\theta.
\end{eqnarray*}
Therefore, by \eqref{eq:mu_s},
\begin{eqnarray} \label{eqn:mucalc1} 
\lefteqn{\lambda \mu^s_{c}\big(\{(x,\alpha): \scap(x_0,2^{-n})\subset \scap(x,\alpha)\}\big)}\\
& & =\lambda\,\int_{2^{-n}}^{c\pi/\max(2,c)} \int_{\BS^{d-1}} {\bf 1}\{\scap(x_0,2^{-n})\subset \scap(x,\alpha)\} g_c(\alpha)\,
\dint x \, \dint\alpha \nonumber \\
& & \ge \lambda\,\int_{2^{-n}}^{c\pi/2}\sigma_{d-1}(\{x\in \BS^{d-1}: \scap(x_0,2^{-n})\subset \scap(x,\alpha)\})
g_c(\alpha)\, \dint\alpha \nonumber \\
& & \geq \lambda\omega_{d}
\int_{2^{-n}}^{2c/\sqrt{d}} \int_0^{\alpha-2^{-n}} \sin^{d-2} \theta \, \dint\theta
\,g_c(\alpha) \,\dint\alpha. \nonumber
\end{eqnarray}
\noindent 
Let $f(\alpha):=\int_0^\alpha \sin^{d-2} \theta \dint\theta.$ It is well 
known that for any $0<\theta<\pi/2$ we have that 
$
\theta-\theta^3<\sin \theta < \theta, 
$
and so 
\begin{eqnarray} \label{eqn:falphalower}
\lefteqn{f(\alpha)\geq \int_0^\alpha (\theta-\theta^3)^{d-2} \dint\theta
\geq \int_0^\alpha (\theta-\alpha^3)^{d-2} \dint\theta}\\
& & =\left[\frac{(\theta-\alpha^3)^{d-1}}{d-1}\right]_0^\alpha
=\frac{\alpha^{d-1}}{d-1}(1-\alpha^2)^{d-1}
\geq \frac{\alpha^{d-1}}{d-1}(1-(d-1)\alpha^2)
=\frac{\alpha^{d-1}}{d-1}-\alpha^{d+1}. \nonumber
\end{eqnarray}
\allowdisplaybreaks
Thus, combining \eqref{eqn:mucalc1} and \eqref{eqn:falphalower} 
we see that
\begin{eqnarray*} 
\lefteqn{\lambda \mu^s_{c}\big(\{(x,\alpha): \scap(x_0,2^{-n})\subset \scap(x,\alpha)\}\big)}\\& & \geq \lambda\omega_{d} 
\int_{2^{-n}}^{2c/\sqrt{d}} \left(\frac{(\alpha-2^{-n})^{d-1}}{d-1} 
-(\alpha-2^{-n})^{d+1}\right)g_c(\alpha) \,\dint\alpha \nonumber \\
& & = \lambda\omega_{d} 
\int_{2^{-n}}^{2c/\sqrt{d}} (\alpha-2^{-n})^{d-1}
\left(\frac{1}{d-1} 
-(\alpha-2^{-n})^{2}\right)g_c(\alpha) \,\dint\alpha \nonumber \\
& & \geq \lambda\omega_{d} 
\int_{2^{-n}}^{2c/\sqrt{d}} (\alpha-2^{-n})^{d-1}
\left(\frac{1}{d-1} -\frac{1}{d}\right)g_c(\alpha)\, \dint\alpha \nonumber \\
& & \geq \frac{\lambda\omega_{d}}{d(d-1)}
\int_{2^{-n}}^{2c/\sqrt{d}}(\alpha-2^{-n})^{d-1} g_c(\alpha) \,\dint\alpha.
\end{eqnarray*}
In addition, using the lower bound in \eqref{eq:gcestimate} with $\alpha_0={2c\over \sqrt{d}}$ we get that 
for every $\alpha<{2c\over \sqrt{d}}$ it holds that
\[
g_c(\alpha)\geq c^{d-1}\cos(2/\sqrt{d})^{d-2}\alpha^{-d}.
\]
From this we conclude that
\begin{eqnarray} \label{eqn:mucalc2}
\lefteqn{\lambda \mu^s_{c}\big(\{(x,\alpha): \scap(x_0,2^{-n})\subset \scap(x,\alpha)\}\big)}\\
& & \geq \lambda C(d,c)\int_{2^{-n}}^{2c/\sqrt{d}}\left(\alpha-2^{-n}\right)^{d-1} 
\alpha^{-d} \,\dint\alpha \nonumber\\
& & \geq \lambda C(d,c)\int_{2^{-n}}^{2c/\sqrt{d}} 
\left(\alpha^{d-1}-(d-1)2^{-n}\alpha^{d-2}\right)
\alpha^{-d} \,\dint\alpha \nonumber \\
& &  =\lambda C(d,c)\Big([\log \alpha]_{2^{-n}}^{2c/\sqrt{d}}
-(d-1)2^{-n}[-\alpha^{-1}]_{2^{-n}}^{2c /\sqrt{d}} \Big)
\geq \lambda C(d,c)n, \nonumber
\end{eqnarray}
for every $n$ satisfying $2^{n/2}>c^{-1}\sqrt{d}e^{d-1}$. This can be achieved by choosing $N$ as the maximum of $1$ and ${2\over\log 2}\log(c^{-1}\sqrt{d}e^{d-1})$.
Next, we let 
\[
M_n:=\big\{x\in S_n: \not \exists (y,\alpha)\in \Psi(\lambda \mu^s_{c}) 
\textrm{ with } \scap(x,2^{-n})\subset \scap(y,\alpha)\big\}.
\]
Thus, $M_n$ is the set of caps with centers in $S_n$ and of radii 
$2^{-n}$ which are not covered by one single cap from the Poisson 
process $\Psi(\lambda \mu^s_{c}) .$
Observe that if $M_n=\emptyset,$
then $\BS^{d-1} =\CC_n(\lambda \mu^s_{c})$, since
\begin{equation} \label{eqn:SsubsetSC}
\BS^{d-1}\subset \bigcup_{x\in S_n} \scap(x,2^{-n}).
\end{equation}
Indeed, this follows because otherwise 
$
y\in \BS^{d-1} \setminus \bigcup_{x\in S_n} \scap(x,2^{-n})  
$
would be a point in $\BS^{d-1}$ such that the spherical distance between $y$ and any 
$x\in S_n$ would be at least $2^{-n},$ contradicting the maximality of $S_n.$  
Next, we see that 
\begin{eqnarray*}
\lefteqn{\BP(\CC(\lambda \mu^s_{c})^c \neq\emptyset)\leq \BP(\CC_n(\lambda \mu^s_{c})^c  \neq \emptyset)
\leq \BP(|M_n|>0)\leq \BE[|M_n|]}\\
& & =|S_n| \BP(\not \exists (y,\alpha)\in \Psi_n(\lambda \mu^s_{c}):
\scap(x_0,2^{-n})\subset \scap(y,\alpha))\\
& & =|S_n|\exp\left(- 
\lambda \mu^s_{c}\big(\{(x,\alpha): \scap(x_0,2^{-n})\subset \scap(x,\alpha)\}\big) \right) \\
& & \leq C(d) 2^{(d-1)n}\exp(-\lambda C(d,c)n),
\end{eqnarray*}
by using Lemma \ref{lemma:cardSnest} and \eqref{eqn:mucalc2}, and by taking 
$n$ large enough.
Thus, for $\lambda>2(d-1)/C(d,c)$ we see that 
\[
\BP(\CC(\lambda \mu^s_{c})^c \neq \emptyset) 
\leq \lim_{n\to \infty} C(d) \exp((-\lambda C(d,c) +(d-1)\log 2)n)=0.
\]
This completes the argument.
\end{proof}

\begin{proof}[Proof of Theorem \ref{thm:cover}]
According to \eqref{eqn:C-inclusion}, we have that 
$\CC(\lambda\mu^{s}_{c,f_1})\subset\CC(\lambda \mu^{s,-})$
and so 
\[
\BP(\CC(\lambda \mu^{s,-})=\BS^{d-1})
\geq \BP(\CC(\lambda\mu^{s}_{c,f_1})=\BS^{d-1}),
\]
where $f_1$ is as explained just above \eqref{eqn:C-inclusion}.
Furthermore, according to Lemma \ref{lemma:measure}, 
$\mu^{s}_{c,f_1}$ is of the form \eqref{eq:mu_s} with $c={1\over 2}\arccos(\cosh^{-1}(1))$, and so the statement
follows by applying Theorem \ref{thm:cover_gen}.
\end{proof}

\section{Proof of Theorem \ref{thm:Cbounded}.} \label{sec:fracball}

The purpose of this section is to prove Theorem \ref{thm:Cbounded}.
Recall therefore the fractal ball model, and in particular definition
\eqref{eqn:muEdef} of the intensity measure $\mu^E.$ For the purpose of this section
and the results that we are proving, it is natural to work in $\BR^d.$ 
However, we warn the reader that we shall later apply these results in $\BR^{d-1}.$

In order to proceed, we will consider 
crossings of annuli $\CA(a,b):=[-b,b]^d \setminus [-a,a]^d$ 
where $0<a<b<\infty.$ The event $\Cr(\CA(a,b),\CC(\lambda\mu^E))$
that the annulus $\CA(a,b)$ is crossed by the set $\CC(\lambda\mu^E)$
is then formally defined in the following way:
\[
\Cr(\CA(a,b),\CC(\lambda\mu^E))
=\{\exists \textrm{ a path } \omega: \omega \subset \CC(\lambda\mu^E),
\omega\cap \partial [-a,a]^d\neq \emptyset \textrm{ and } 
\omega\cap \partial [-b,b]^d\neq \emptyset \}.
\]
Here, by a path we understand a union of balls $\bigcup_{i=1}^m B(x_i,r_i)$, such that 
$(x_i,r_i)\in \Psi(\lambda\mu^E)$, $1\leq i\leq m$ and 
$B(x_i,r_i)\cap B(x_{i+1},r_{i+1})\neq \emptyset$, $1\leq i\leq m-1$. 
If $\Cr(\CA(a,b),\CC(\lambda\mu^E))$ does not occur, we say that 
$[-a,a]^d$ is {\em separated} from $[-b,b]^d$ (by the set 
$\CV(\lambda\mu^E)$, defined in \eqref{eqn:defVfracball}).
Theorem \ref{thm:Cbounded} follows from the following result.
\begin{theorem} \label{thm:anncrossings}
There exists $\lambda>0$ such that for any $0<a<b<\infty$ we have that 
\[
\BP(\Cr(\CA(a,b),\CC(\lambda\mu^E)))<1.
\]
\end{theorem}

\begin{proof}[Proof of Theorem \ref{thm:Cbounded} from 
Theorem \ref{thm:anncrossings}]
Let $\lambda>0$ be such that the conclusion of Theorem 
\ref{thm:anncrossings} holds, and let $a=1,b=3.$ Furthermore,
let 
\begin{equation} \label{eqn:psidef}
\psi:=\BP(\Cr(\CA(1,3),\CC(\lambda\mu^E)))<1.
\end{equation}
Observe next that the events 
$(\Cr(\CA(3^{2n},3^{2n+1}),\CC(\lambda\mu^E))_{n\geq 0}$ are independent. Indeed,
since the balls in $\Psi(\lambda \mu^E)$
all have radius at most $1$ the event $\Cr(\CA(3^{2n},3^{2n+1}),\CC(\lambda\mu^E))$ depends only on the restriction of $\Psi(\lambda\mu^E)$ to the set 
$\{(x,r)\colon x\in(-3^{2n+1}-1,3^{2n+1}+1)^d\setminus (-3^{2n}-1,3^{2n}+1)^d\}$. 
All these sets are disjoint for $n\in\BN$ and the statement follows by 
the independence property of Poisson  process \cite[Theorem 5.2]{LP18}. 
Observe further that the collection of balls generated by $\Psi(\lambda \mu^E)$ 
is semi-scale invariant as described in Section \ref{sec:Notationfracball}. From this 
it follows  that 
\begin{equation}\label{eqn:crossmon}
\BP(\Cr(\CA(3^{2n+2},3^{2n+3}),\CC(\lambda\mu^E)))
\leq 
\BP(\Cr(\CA(3^{2n},3^{2n+1}),\CC(\lambda\mu^E))),
\end{equation}
for every $n\geq 0$, where we additionally note that discarding some balls may only decrease the probability of crossing. 

Recall from Section \ref{sec:Notationfracball} that 
$o\in \CC(\lambda\mu^E)$ almost surely. Then, 
let $\CC_o(\lambda\mu^E)$ denote the cluster containing the origin $o$, and note that 
\begin{eqnarray} \label{eqn:Cobounded}
\lefteqn{\BP(\CC_o(\lambda\mu^E) \textrm{ is unbounded })
\leq \BP\left(\bigcap_{n=0}^\infty 
\Cr(\CA(3^{2n},3^{2n+1}),\CC(\lambda\mu^E))
\right)}\\
& & = \prod_{n=0}^\infty \BP\left(
\Cr(\CA(3^{2n},3^{2n+1}),\CC(\lambda\mu^E))
\right)
\leq \prod_{n=0}^\infty \psi=0, \nonumber
\end{eqnarray}
by using independence in the first equality, and \eqref{eqn:crossmon}
and \eqref{eqn:psidef} in the last inequality.
Since there are countably many balls that make up $\CC(\lambda\mu^E),$ 
it follows that there are at most countably 
many connected components in $\CC(\lambda\mu^E).$ Therefore, 
Theorem \ref{thm:Cbounded} follows from \eqref{eqn:Cobounded}.
\end{proof}

We will now address Theorem \ref{thm:anncrossings}. We also note that 
we will use the result of Theorem \ref{thm:anncrossings} to prove 
Theorem \ref{thm:disconnect}. 
Instead of  proving Theorem \ref{thm:anncrossings} from first 
principles, which would require a long and technical argument, 
we will rely on a previous result from \cite{O96} concerning the so-called 
Mandelbrot fractal percolation model. In order to use this result, we shall 
first need to describe the model. For simplicity and brevity, we will be
somewhat informal. 
Let $0<p<1$ and fix some integer $M>1.$ 
Consider $[0,1]^d$ and divide it into $M^d$ closed boxes of side 
length $M^{-1}$ in the canonical way (so that only boundaries of neighboring 
boxes overlap). These boxes will be referred to as level-1 boxes
and we denote this collection by $\CX^{M,1}.$ A box in $\CX^{M,1}$ will 
be denoted by $X$ so that $X\in \CX^{M,1}.$ 
For each $X\in \CX^{M,1},$ retain $X$ independently with 
probability $p$ and otherwise discard it. Then, let $D^{M,1}$ be the 
set of 
boxes that were retained. By letting 
\[
\CD^{M,1}:=\bigcup_{X\in D^{M,1}} X,
\]
we see that $\CD^{M,1}$ is almost surely a random closed set.

In the second step, we again divide $[0,1]^d$ into boxes, but this 
time we divide it into $M^{2d}$ closed boxes of side 
length $M^{-2}$ in the canonical way. Again, we use $X$ to denote a box, and the collection 
$\CX^{M,2}$ will be referred to as level-2 boxes. Then, we independently 
retain each box with probability $p,$ and otherwise discard it.
The set of retained level-2 boxes is denoted by $D^{M,2},$
and we let  
\[
\CD^{M,2}:=\bigcup_{X\in D^{M,2}} X
\]
be the corresponding random closed set.

This procedure is repeated for $n=3,4,\ldots$ and we obtain 
a sequence $(\CD^{M,n})_{n\geq 1}$ of random closed sets. We can then 
define
\[
\CD^M:=\bigcap_{n=1}^\infty \CD^{M,n},
\]
which is a random fractal set. The latter is well defined by compactness. 
Note that $\CD^M=\CD^M(p)$ and that $\CD^M$ is a closed set.

In \cite{O96} it was 
proven that for any $M<\infty$ large enough, there exists $p<1$, 
such that with positive probability 
$\CD^M$ contains a connected component
(in this case a ``sheet'') separating $\{0\}\times [0,1]^{d-1}$ from 
$\{1\}\times [0,1]^{d-1}.$ To be precise, the latter means that there is no
path $\omega\subset [0,1]^d\cap (\CD^M)^c$ which intersects both 
$\{0\}\times [0,1]^{d-1}$ and $\{1\}\times [0,1]^{d-1}.$
Furthermore, it is easy to modify the argument in \cite{O96} to work 
for other geometries. In particular, the same proof works for showing
that with positive probability, $\{0\}\times [0,1]^{d-1}$ is separated 
from $\{1/3\}\times [0,1]^{d-1}$ by a sheet in $\CD^M.$ That is, there is 
a sheet separating the left-hand side from the right-hand side of the 
rectangle $[0,1/3]\times [0,1]^{d-1}$ according to the ordering with respect to the first coordinate axis. 
Let $M<\infty$ and $p<1$ be such that this separating event has positive 
probability and consider the annulus $[0,1]^d \setminus [1/3,2/3]^d.$
Then, let 
$R_1,\ldots,R_{2d}$ denote the overlapping and congruent copies of 
$[0,1/3]\times [0,1]^{d-1}$ that can be used to cover this annulus
in the natural way. That is, we let 
$R_1=[0,1/3]\times [0,1]^{d-1},
R_2=[2/3,1]\times [0,1]^{d-1}, R_3=[0,1]\times[0,1/3]\times [0,1]^{d-2}$
etc. Let $\Sep(R_n, \CD^M)$ denote the event that $(\CD^M)^c$ 
does not cross the rectangle $R_n$ in the short direction, and 
let $\Sep([1/3,2/3]^d,[0,1]^d,\CD^M)$ denote the event 
that $(\CD^M)^c$ does not cross the annulus 
$[0,1]^d \setminus [1/3,2/3]^d$ (see Figure \ref{fig:separation} for 
an illustration of these events).
\begin{figure}
\centering
\begin{subfigure}{.4\textwidth}
  \centering
  \includegraphics[scale=0.5,trim= 45mm 20mm 240mm 10mm, clip, page=1 ]{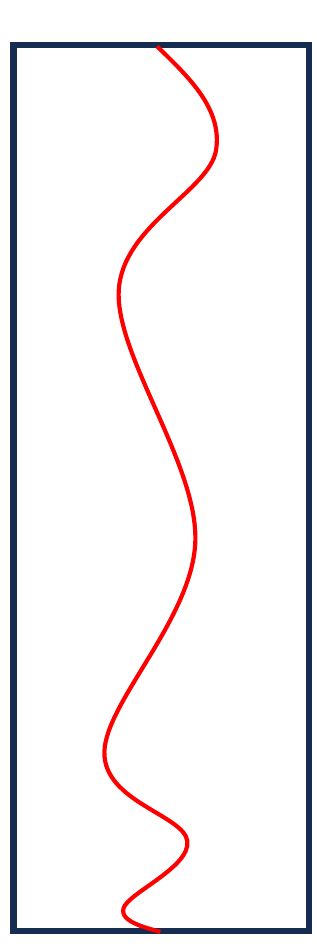} %LBRT
  \captionsetup{width=.8\linewidth}
  \caption{\small A sheet (red) in $\CD^M$ blocks a crossing in $(\CD^M)^c$
  of the rectangle in the short direction.}
\end{subfigure}%
\begin{subfigure}{.6\textwidth}
  \centering
  \includegraphics[scale=0.5,trim= 60mm 30mm 140mm 10mm, clip, page=2 ]{Sheets.pdf}
    \captionsetup{width=.8\linewidth}
  \caption{\small All four separation events in $R_1,R_2,R_3$ and $R_4$ occur so that 
  the annulus cannot be crossed by $(\CD^M)^c$.}
  %\label{fig:sub2}
\end{subfigure}
\caption{}
\label{fig:separation}
\end{figure}
We see that by the FKG inequality \cite[Theorem B15]{Liggett},
\begin{equation}\label{eqn:sepprobab}
\BP(\Sep([1/3,2/3]^d,[0,1]^d,\CD^M))
\geq \prod_{n=1}^{2d}\BP(\Sep(R_n,\CD^M))
=\BP(\Sep(R_1,\CD^M))^{2d}>0,
\end{equation}
where the equality follows from rotational and 
translational invariance, and the last inequality by our choice of $M$ and $p$.
We note that the application of the FKG inequality uses that the events 
$\Sep(R_n,\CD^M)$ are decreasing in the configuration of boxes. Indeed, 
if a crossing by $(\CD^M)^c$ occurs, then retaining more boxes may only destroy 
this crossing.

We are now ready to prove Theorem \ref{thm:anncrossings}.
\begin{proof}[Proof of Theorem \ref{thm:anncrossings}]
Let $M<\infty$ and $p<1$ be close enough to $1$ so that \eqref{eqn:sepprobab} holds, and consider them now to be fixed. Our goal is 
to prove that for $\lambda>0$ small enough, we can couple the 
fractal ball model with the Mandelbrot fractal percolation 
model such that with probability one,
\begin{equation} \label{eqn:CCD}
\CD^M(p) \subset \CV(\lambda \mu^E).
\end{equation}
Note that while $\CD^M(p) \subset [0,1]^d,$ it is not the case that
$\CV(\lambda \mu^E)\subset [0,1]^d$. We will in fact only study 
$\CV(\lambda \mu^E)\cap [0,1]^d$ but this is not reflected in 
\eqref{eqn:CCD} as it is not necessary.
Assuming \eqref{eqn:CCD}, we conclude that 
\begin{eqnarray}\label{eqn:crossprobbasic}
\lefteqn{\BP(\Cr(\CA(1/6,1/2),\CC(\lambda\mu^E)))}\\
& & =1-\BP(\Sep([-1/6,1/6]^d,[-1/2,1/2]^d,\CV(\lambda\mu^E))) \nonumber \\
& & =1-\BP(\Sep([1/3,2/3]^d,[0,1]^d,\CV(\lambda\mu^E))) \nonumber \\
& & \leq 1-\BP(\Sep([1/3,2/3]^d,[0,1]^d,\CD^M(p)))<1, \nonumber
\end{eqnarray}
where we used translational invariance in the second equality,
\eqref{eqn:CCD} in the first inequality and \eqref{eqn:sepprobab}
in the last inequality. Here, we let
$\Sep([-1/6,1/6]^d,[-1/2,1/2]^d,\CV(\lambda\mu^E))$ and
$\Sep([1/3,2/3]^d,[0,1]^d,\CV(\lambda\mu^E))$ 
denote separation events as indicated by the notation, analogous to the 
separation events defined above for $\CD^M(p)$.

From this, the statement follows by a straightforward tiling and scaling
argument as we now explain (see also Figure \ref{Fig:Tiling} for an illustration 
of the idea). Similar to \eqref{eqn:Cdeffracball}, let for 
$\epsilon>0,$
\[
\left(\CC(\lambda\mu^E)\right)_\epsilon
=\bigcup_{(x,r)\in \Psi(\lambda\mu^E): r\leq \epsilon} B(x,r),
\]
which is the set covered by the balls in $\Psi(\lambda\mu^E)$ with radius at 
most $\epsilon.$ We note that by the scaling invariance described in Section 
\ref{sec:Notationfracball} we have that for any $N\geq 0,$
\begin{equation} \label{eqn:annscaling}
\BP(\Cr(\CA(3^{-N-1}/2,3^{-N}/2),(\CC(\lambda\mu^E))_{3^{-N}}))
=\BP(\Cr(\CA(1/6,1/2),\CC(\lambda\mu^E))).
\end{equation}
Next, let 
\[
\CE_N=\{\not \exists (x,r)\in \CC(\lambda\mu^E): 
B(x,r)\cap \CA(3^{-N-1}/2,3^{-N}/2) \neq  \emptyset, r>3^{-N}\}.
\]
In words, $\CE_N$ is the event that there are no balls of the Poisson process of radius larger than 
$3^{-N}$ touching the annulus $\CA(3^{-N-1}/2,3^{-N}/2).$
Next, we observe that 
\begin{align}
\BP(\Cr(&\CA(3^{-N-1}/2,3^{-N}/2),\CC(\lambda\mu^E)))\notag\\
& \leq \BP(\Cr(\CA(3^{-N-1}/2,3^{-N}/2),\CC(\lambda\mu^E))|\CE_N)\BP(\CE_N)
+\BP(\CE_N^c) \notag\\
 & = \BP(\Cr(\CA(3^{-N-1}/2,3^{-N}/2),(\CC(\lambda\mu^E))_{3^{-N}})\BP(\CE_N)
+\BP(\CE_N^c) \notag\\
 & =\BP(\Cr(\CA(1/6,1/2),\CC(\lambda\mu^E)))\BP(\CE_N)
+\BP(\CE_N^c)<1.\label{eq:19_12_24}
\end{align}
where we used the independence and superposition properties of Poisson processes (see Theorem 3.3 in \cite{LP18}) in the penultimate equality, 
\eqref{eqn:annscaling} in the last equality 
and \eqref{eqn:crossprobbasic} in the last inequality.
{Next, note that for fixed $a<b,$ one can first choose $N<\infty$
so that $b-a>3^{-N}$ and then consider a set $\mathcal{T}_N$ of copies of 
the annulus $\CA(3^{-N-1}/2,3^{-N}/2)$, such that the annuli are centred
along the boundary of the box $\left[-\frac{a+b}{2},\frac{a+b}{2}\right]^d$ and are covering the boundary. 
If none of these annuli are crossed by $\CC(\lambda\mu^E)$, 
then neither is $\CA(a,b)$. Hence, we have
\begin{align*}
\BP(\Cr(\CA(a,b),\CC(\lambda\mu^E))^c)&\ge \BP(\text{None of annuli from $\mathcal{T}_N$ are crossed})\\
&\ge \BP(\Cr(\CA(3^{-N-1}/2,3^{-N}/2),\CC(\lambda\mu^E))^c)^{|\mathcal{T}_N|}>0,
\end{align*}
as follows by the FKG inequality \cite[Theorem B15]{Liggett} together with the translation invariance of the process $\Psi(\lambda\mu^E)$ and \eqref{eq:19_12_24}. Application of the FKG inequality is justified by the fact that the event $\Cr(\CA(a,b),\CC(\lambda\mu^E))^c$ is decreasing in the configuration of boxes.}

\begin{comment}
From this, the statement follows by a straightforward tiling and scaling
argument so we will be somewhat informal. To see how the argument works, 
first note that for any $N\geq 1,$ with positive probability,
$\CA(3^{-N-1},3^{-N})$ is not touched by any balls in $\CC(\lambda\mu^E)$ 
of radius larger than $3^{-N}.$ Conditioned on this event, it follows 
by the semi-scale invariance of Section \ref{sec:Notationfracball} that
the (conditional) probability of  
$\Cr(\CA(3^{-N-1},3^{-N}),\CC(\lambda\mu^E))$ can be bounded by 
$\BP(\Cr(\CA(1/6,1/2),\CC(\lambda\mu^E))$.  
Next, note that for fixed $a<b,$ one can first choose $N<\infty$
so that $b-a>3^{-N}$ and then consider a tiling by copies of 
the annulus $\CA(3^{-N-1},3^{-N})$ such that the annulii are centered
along the boundary of the box $\left[-\frac{a+b}{2},\frac{a+b}{2}\right]^d$. 
If none of these annulii are crossed by $\CC(\lambda\mu^E)$, 
then neither is $\CA(a,b)$. Since we saw that 
$\BP(\Cr(\CA(3^{-N-1},3^{-N}),\CC(\lambda\mu^E)))<1$ the statement follows.
The details are 
left to the reader, but see Figure \ref{Fig:Tiling} for an illustration 
of the idea.
\end{comment}

\begin{figure}[t]
	\begin{center}
		\includegraphics[scale=0.5,trim= 20mm 20mm 40mm 10mm, clip ]{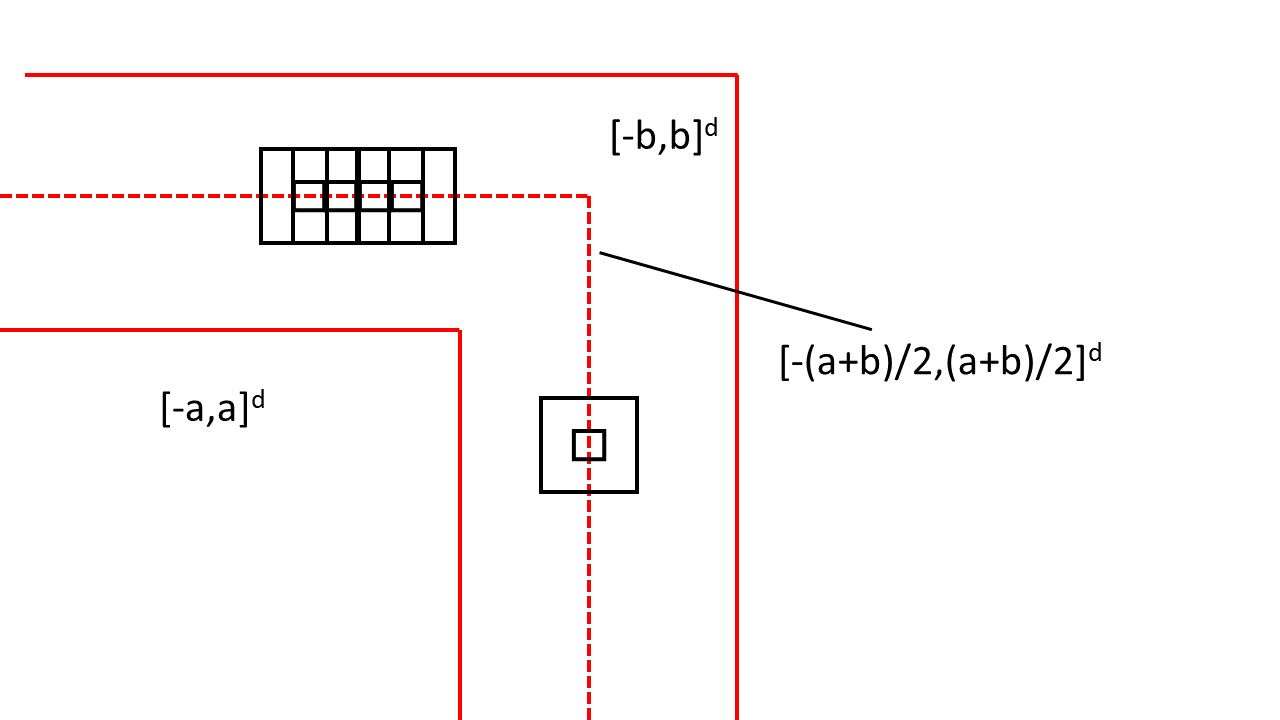} %LBRT
	\end{center}
	\caption{\small A depiction of the tiling. The solid red segments are parts
	of $\partial [-a,a]^d$ and $\partial [-b,b]^d,$ while the dashed segment is a part of $\partial [-(a+b)/2,(a+b)/2]^d$. The single black annulus
	is a copy of $\CA(3^{-N-1}/2,3^{-N}/2),$ while on the top
	dashed line we see four such copies next to each other. If the entire
	boundary $\partial [-(a+b)/2,(a+b)/2]^d$ is tiled in this way, a 
	crossing of  $\CA(a,b)$  implies that there exists a crossing of 
	at least one of the copies of $\CA(3^{-N-1}/2,3^{-N}/2).$ We note that 
	there may be some technicalities involved when performing the tiling at the corners,
	but these are easily resolved. It is also easy to generalize this 
	to any dimension $d\geq 2.$ }
	\label{Fig:Tiling}
\end{figure}

We therefore turn to establishing \eqref{eqn:CCD}.
To that end, let 
\begin{equation} \label{eqn:Psin}
\left(\Psi(\lambda \mu^E)\right)_{n-1}^{n}
:=\{(x,r)\in \Psi(\lambda \mu^E): M^{-n}<r\leq M^{-n+1} \},
\end{equation}
for some $M>1$ and then let 
\[
\left(\CV(\lambda \mu^E)\right)_{n-1}^{n}
=\BR^d \setminus 
\left(\bigcup_{(x,r)\in \left(\Psi(\lambda \mu^E)\right)_{n-1}^{n}} B(x,r)
\right)
\]
so that 
\[
\CV(\lambda \mu^E)
=\bigcap_{n=1}^\infty \left(\CV(\lambda \mu^E)\right)_{n-1}^{n}.
\]
Next, define the sequence $\tilde{\CD}^{M,1},\tilde{\CD}^{M,2}, \ldots$  
in the following way. Let 
\[
\tilde{D}^{M,n}:=\{X\in \CX^{M,n}: 
\not \exists (x,r)\in \left(\Psi(\lambda \mu^E)\right)_{n-1}^{n}
\textrm{ such that } B(x,r)\cap X \neq \emptyset\},
\]
so that $\tilde{D}^{M,n}$ consists of those level-$ n $ boxes which 
are untouched by the balls corresponding to 
$\left(\Psi(\lambda \mu^E)\right)_{n-1}^{n}.$ If we let 
\[
\tilde{\CD}^{M,n}:=\bigcup_{X\in \tilde{D}^{M,n}}X, 
\]
then clearly we have that almost surely
\[
\left(\CV(\lambda \mu^E)\right)_{n-1}^{n}\supset  \tilde{\CD}^M_{n},
\]
and since this holds for any $n=1,2,\ldots$ we conclude that 
\begin{equation} \label{eqn:VEtildeDdom}
\CV(\lambda \mu^E)\supset  \tilde{\CD}^M
:=\bigcap_{n=1}^\infty\tilde{\CD}^{M,n}
\end{equation}
with probability one.
We note that $\tilde{\CD}^M$ is similar to the Mandelbrot 
fractal percolation model, but in contrast to that model, we here 
have dependencies for different $X\in \CX^{M,n}$ since a ball 
corresponding to $(x,r)\in\left(\Psi(\lambda \mu^E)\right)_{n-1}^{n}$ can 
intersect multiple boxes $X\in\CX^{M,n}$. We will now address this through 
a domination argument.

First, we note that for any $n$ and any $X\in \CX^{M,n}$ 
the probability
\[
\BP\left(\exists (x,r)\in \left(\Psi(\lambda \mu^E)\right)_{n-1}^{n}
\textrm{ such that } B(x,r)\cap X \neq \emptyset \right),
\]
is the same. The non-dependency on $n$ follows by scaling invariance as explained in Section \ref{sec:Notationfracball}, while the non-dependency on $X\in \CX^{M,n}$ follows 
from translation invariance. It therefore makes sense to define
\begin{equation} \label{eqn:tildepdef}
\tilde{p}
:=\BP\left(\exists (x,r)\in \left(\Psi(\lambda \mu^E)\right)_{n-1}^{n}
\textrm{ such that } B(x,r)\cap X \neq \emptyset \right).
\end{equation}
Next, as observed, although the marginal probabilities that 
$X,X'\in \CX^{M,n}$ belong to $\tilde{\CD}^M_n$ are both $\tilde{p},$ 
the events are not independent whenever $X,X'$ are close.
However, the events are independent whenever the distance between 
$X,X'$ is larger than $2M^{-n+1},$ since 
$\left(\Psi(\lambda \mu^E)\right)_{n-1}^{n}$ only contains 
balls of radius at most $M^{-n+1}.$
Because of this finite range dependency, we can use
\cite[Theorem B26]{Liggett}
to see that there exists some $0<p<1$ and 
a coupling between $(\tilde{\CD}^{M,n}(\tilde{p}))_{n\geq 1}$ 
and $(\CD^{M,n} (p))_{n\geq 1}$ such that 
\begin{equation} \label{eqn:tildeDdom1}
\tilde{\CD}^{M,n}(\tilde{p}) \supset \CD^{M,n} (p),
\end{equation}
almost surely for $n=1,2,\ldots$ Furthermore, by the same theorem, for any fixed 
$p<1,$ we can make \eqref{eqn:tildeDdom1} hold by taking
$\tilde{p}$ large enough. 
Thus, for our choice of $p$ (that we made
at the start), we let $\tilde{p}$ be determined so 
that \eqref{eqn:tildeDdom1}
holds. Then, we let $\lambda>0$ be such that \eqref{eqn:tildepdef} 
holds. We conclude that almost surely
\begin{equation} \label{eqn:tildeDdom2}
\tilde{\CD}^M(\tilde{p}) 
=\bigcap_{n=1}^\infty\tilde{\CD}^{M,n}(\tilde{p})
\supset \bigcap_{n=1}^\infty \CD^{M,n}(p)
=\CD^M (p). 
\end{equation}
Thus, combining \eqref{eqn:VEtildeDdom} and \eqref{eqn:tildeDdom2} 
we can conclude \eqref{eqn:CCD}.
\end{proof}

\section{Proof of Theorem \ref{thm:disconnect}}\label{sec:Proof2}
We now have all the main tools needed to prove Theorem 
\ref{thm:disconnect}. However, there are several technical 
details that needs to be addressed, and therefore this section will 
consist of a number of incremental results. 
The general idea is to project the caps induced by the process 
$\Psi(\lambda \mu_r^{s,+})$, defined in Section \ref{sec:inducedBool},
from the surface of the sphere onto $\BR^{d-1}.$ We then want to use
our results about the fractal ball model in order to draw a conclusion 
about the lack of connectivity of $\CC(\lambda \mu_r^{s,+}).$ We note 
that until now, the fractal ball model has been studied 
in $\BR^{d},$ while here it will ``live'' in $\BR^{d-1}.$ Therefore, 
while $r^{-d-1} \dint r$ in \eqref{eqn:muEdef} yields
a semi-scale invariant fractal ball model as discussed in Section 
\ref{sec:Notationfracball}, the corresponding factor here  is
$r^{-d} \dint r.$

Recall the injective projection 
$\Pi_{\rm vert}: \BS^{d-1}_{+} \to B^{d-1}(o,1)$ from 
Section \ref{sec:Notation}.
In order to understand the technical difficulties we are facing, 
and in order to explain below why we are doing certain things,
we note somewhat informally that 
\begin{itemize}
\item[(a)] A cap on $\BS_+^{d-1}$ which is projected onto $\BR^{d-1}$ using $\Pi_{\rm vert}$ is 
not a ball, but an ellipsoid.

\item[(b)] The collection of projected center points
 $\{\Pi_{\rm vert}(x): (x,\alpha) \in \Psi(\lambda \mu_r^{s,+})\}$
will be a subset of $B^{d-1}(o,1).$ In contrast, the fractal ball model 
defined in Section \ref{sec:Notationfracball} was defined in all
of $\BR^{d-1}$.

\item[(c)] The collection of projected center points will look
qualitatively different in the neighbourhood of $o\in \BR^{d-1}$ compared to the neighbourhood of
the boundary $\partial B^{d-1}(o,1).$ 
 
\end{itemize}

Below, we will prove Theorem \ref{thm:disconnect2} from which Theorem 
\ref{thm:disconnect} will easily follow. In order to state Theorem 
\ref{thm:disconnect2}, let 
\begin{equation} \label{eqn:mucsdef}
\widetilde\mu^s_{g}(\dint (x,\alpha)):=\sigma_{d-1}(\dint x) 
\times  {\bf 1}\left(0<\alpha \leq C_d\right)g(\alpha)
\dint \alpha,
\end{equation}
where $C_d:=(2\sqrt{2(d-1)})^{-1}$ and $g(\alpha)$ is
such that $C'\alpha^{-d}\leq g(\alpha)\leq C\alpha^{-d}$ for some $C',C\in(0,\infty)$. 
We then define the Poisson  process 
$\Psi(\lambda \widetilde\mu^s_{g})$ on 
$\BS^{d-1} \times (0,C_d]$ as usual
along with the set $\CC(\lambda \widetilde\mu^s_{g}).$ As will become clear later, 
the reason for introducing this class of measures is to address 
$(b)$ and $(c)$ in the list of technical difficulties above.

\begin{theorem} \label{thm:disconnect2}
For any measure $\widetilde\mu^s_{g}$ defined as above,
we have that for $\lambda>0$ small enough, 
\[
\BP(\CC(\lambda \widetilde\mu^s_{g})\textrm{ is not connected})>0.
\]
\end{theorem}
Before addressing the proof of 
Theorem \ref{thm:disconnect2}, we will demonstrate how 
Theorem \ref{thm:disconnect} follows from it.
\begin{proof}[Proof of Theorem \ref{thm:disconnect} from 
Theorem \ref{thm:disconnect2}]
By Lemma \ref{lemma:measure}, the intensity measure 
$\lambda\mu_{c,f}^{s}$, defined by \eqref{eq:MuFC} with $f\equiv 0$ and $c=2\sinh(1)+1$
is of the form \eqref{eqn:mucsdef}, but with 
${\bf 1}(0<\alpha\leq C_d)$ replaced by 
${\bf 1}\left(0<\alpha\leq c\pi/\max(2,c)\right)$. 
However, this is easily circumvented 
as follows. Let 
\[
G=\left\{\not \exists (x,r)\in \Psi( \lambda\mu_{c,f}^s): 
r\geq C_d\right\},
\]
and observe that a straightforward calculation shows that $\BP(G)>0.$ 
Furthermore, 
it is an easy consequence of the independence and superposition properties of Poisson processes (see Theorem 3.3 in \cite{LP18})
that conditioning on the event $G$ is equivalent to letting 
$g_c(\alpha)$ be defined on 
$(0,C_d]$ rather than on $(0,c\pi/\max(2,c)].$ Then conditioning on $G$ by Lemma \ref{lemma:measure} we also have
\[
C'=c^{d-1}\cos(C_d/c)^{d-2}\leq g(\alpha)\leq C=c^{d-1}(\pi/2)^d,\qquad\text{ for }\qquad \alpha\in(0,C_d].
\] 
Therefore, by using Theorem \ref{thm:disconnect2} we get
\begin{eqnarray} \label{eqn:notconnmon1}
\lefteqn{\BP(\CC(\lambda\mu_{c,f}^{s}) \textrm{ is not connected})}\\
& &\geq \BP(\CC(\lambda\mu_{c,f}^{s}) \textrm{ is not connected} |G)
\BP(G)
= \BP(\CC(\lambda \mu_{c}^{s}) \textrm{ is not connected})
\BP(G)>0. \nonumber
\end{eqnarray}
Then, we can 
use \eqref{eqn:C+inclusion}  
to conclude that 
\begin{equation} \label{eqn:notconnmon2}
\BP(\CC(\lambda \mu_r^{s,+}) \textrm{ is not connected})
\geq \BP(\CC(\lambda\mu_{c,f}^{s}) \textrm{ is not connected})
\end{equation}
for $r>{1\over 2}\big(\tan\big({e\pi\over e^2-e+1}\big)\big)^{-1}$. Combining \eqref{eqn:notconnmon1} and \eqref{eqn:notconnmon2}
gives us the result.
\end{proof}

Consider now the following variant of the fractal ball model 
where we use the intensity measure $\widetilde\mu^E_g$ on 
\[
\left[-2C_d,2C_d\right]^{d-1}
\times \left(0,C_d\right]
\subset 
\BR^{d-1}\times (0,1],
\]
defined by 
\begin{eqnarray} \label{eqn:mucEdef} 
\widetilde\mu^E_g(\dint (x,r))={\bf 1}\left(x\in 
\left[-2C_d,2C_d\right]^{d-1}\right)
\ell_{d-1}(\dint x) {\bf 1}\left(0<r\leq  C_d\right]
g(r)\dint r,
\end{eqnarray}
where $g(r)\leq C r^{-d}$ for some $C<\infty$  and $\ell_{d-1}$ 
is Lebesgue measure in $\BR^{d-1}$. 
The reason for restricting our attention to the above geometry will 
become clear later on. We then let 
$\Psi(\lambda \widetilde\mu^E_g)$ be as $\Psi(\lambda \mu^E)$ of Section 
\ref{sec:Preliminaries}, but using the intensity measure $\widetilde\mu^E_g$ 
in place of $\mu^E.$ 
As above, we note that here we are dealing with 
a ball process in $\BR^{d-1}$ rather than in $\BR^d,$ and that the
factor $r^{-d-1}\dint r$ in the definition \eqref{eqn:muEdef} 
of $\mu^E$ has been replaced by 
$g(r) \dint r.$ Intuitively, $\Psi(\lambda \widetilde\mu^E_g)$ is 
a fractal ball process 
in $\BR^{d-1}$ restricted to balls with centers in $[-2C_d,2C_d]$ and of radius in the interval 
$\left(0,C_d\right].$ This ball process is not semi-scale 
invariant, but as we will see, it is ``almost'' semi-scale invariant. 
In order to state the following corollary to Theorem 
\ref{thm:anncrossings}, let $\CC(\lambda \widetilde\mu^E_g)$ be defined in the 
obvious way. 

\begin{corollary} \label{corr:anncrossings}
For any  $0<a<b<C_d,$ 
we have that for $\lambda>0$ small enough,
\[
\BP(\Cr (\CA(a,b),\CC(\lambda\widetilde\mu^E_g)))<1.
\]
\end{corollary}
\begin{proof}
Recall, that by our assumptions $g(r)$ satisfies 
\begin{equation} \label{eqn:gcbound}
g(r)\leq C r^{-d},
\end{equation}
for some $C<\infty.$
Let $\tilde{\mu}_C^E$ be defined similarly to 
$\widetilde\mu^E_g$ in \eqref{eqn:mucEdef}, but where $g(r)$ is replaced
by $Cr^{-d}.$ Furthermore, let $\Psi(\lambda\tilde{\mu}_C^E)$ be 
the corresponding 
Poisson  process of caps and let $\CC(\lambda \tilde{\mu}_C^E)$
be defined as usual. Using \eqref{eqn:gcbound}, it 
is an easy consequence of stochastic monotonicity that
there exists a coupling such that 
$\CC(\lambda\widetilde\mu^E_g) \subset \CC(\lambda \tilde{\mu}_C^E)$
with probability one, and so  
\begin{equation} \label{eqn:gcCineq}
\BP(\Cr(\CA(a,b),\CC(\lambda\widetilde\mu^E_g)))
\leq \BP(\Cr(\CA(a,b),\CC(\lambda \tilde{\mu}_C^E))).
\end{equation}
Next, we scale space by a factor of $C_d^{-1},$ i.e. we 
consider the set 
\begin{equation}\label{eqn:scaledset}
\bigcup_{(x,r)\in \Psi(\lambda\tilde{\mu}_C^E)} 
C_d^{-1}B^{d-1}(x,r),
\end{equation}
and note that the radius of the scaled balls are now bounded by 1. 
Note also the factor $C r^{-d} \dint r$ from \eqref{eqn:gcbound} 
which, as explained
at the start of this section, corresponds to a semi-scale invariant 
fractal ball model in $\BR^{d-1}.$ The added constant $C$ just changes 
the intensity from $\lambda$ to $\lambda C$, therefore, the 
set in \eqref{eqn:scaledset}, has the same distribution as
\begin{equation}\label{eqn:Ctildedeffracball}
\tilde{\CC}(\lambda C \mu^E)
:=\bigcup_{\stackrel{(x,r)\in \Psi(\lambda C \mu^E)}
{x\in [-2,2]^{d-1}}} B^{d-1}(x,r),
\end{equation}
where $\Psi(\lambda C \mu^E)$ is as in Sections 
\ref{sec:Notationfracball} and \ref{sec:fracball}.
Therefore, we can conclude that 
\begin{equation} \label{eqn:prelequal1}
\BP(\Cr(\CA(a,b),\CC(\lambda \tilde{\mu}_C^E)))
=\BP(\Cr(\CA(C_d^{-1}a,C_d^{-1}b),
\tilde{\CC}(\lambda C \mu^E)).
\end{equation}

Since after scaling
the radii of the balls are bounded by 1, we can 
conclude that $\tilde{\CC}(\lambda C \mu^E)\cap [-1,1]^{d-1}$ and 
$\CC(\lambda C \mu^E)\cap [-1,1]^{d-1}$ are equal in distribution. 
The reason is simply that even though $\CC(\lambda C \mu^E)$ is defined
by \eqref{eqn:Cdeffracball} (with $d$ here replaced by $d-1$) 
rather than by 
\eqref{eqn:Ctildedeffracball}, no ball $B^{d-1}(x,r)$ with center 
$x \not \in [-2,2]^{d-1}$ can intersect $[-1,1]^{d-1}.$
Therefore, 
\begin{eqnarray}\label{eqn:prelequal2}
\lefteqn{\BP(\Cr(\CA(C_d^{-1}a,C_d^{-1}b),
\tilde{\CC}(\lambda C \mu^E)\cap[-1,1]^{d-1}))}\\
& & =\BP(\Cr(\CA(C_d^{-1}a,C_d^{-1}b),
\CC(\lambda C \mu^E)\cap[-1,1]^{d-1})). \nonumber
\end{eqnarray}
Combining \eqref{eqn:gcCineq}, \eqref{eqn:prelequal1} and 
\eqref{eqn:prelequal2}, we conclude that for any 
$0<a<b<C_d$ there exists $\lambda>0$
small enough so that 
\begin{eqnarray*}
\lefteqn{\BP(\Cr(\CA(a,b),\CC(\lambda\widetilde\mu^E_g)))
\leq \BP(\Cr(\CA(a,b),\CC(\lambda \tilde{\mu}_C^E)))}\\
& & =\BP(\Cr(\CA(C_d^{-1}a,C_d^{-1}b),
\tilde{\CC}(\lambda C \mu^E))) \\
& & =\BP(\Cr(\CA(C_d^{-1}a,C_d^{-1}b),
\tilde{\CC}(\lambda C \mu^E)\cap[-1,1]^{d-1}))\\
& & =\BP(\Cr(\CA(C_d^{-1}a,C_d^{-1}b),
\CC(\lambda C \mu^E)\cap[-1,1]^{d-1}))\\
& & =\BP(\Cr(\CA(C_d^{-1}a,C_d^{-1}b),
\CC(\lambda C \mu^E)))<1,
\end{eqnarray*}
where we used  the fact that the crossing event 
$\Cr(\CA(C_d^{-1}a,C_d^{-1}b)$ only depends on 
$\tilde{\CC}(\lambda C \mu^E)\cap[-1,1]^{d-1}$ (since
$C_d^{-1}b\leq 1$) in the second equality, and
similarly on $\CC(\lambda C \mu^E)\cap[-1,1]^{d-1}$
in the last equality. Finally, Theorem \ref{thm:anncrossings} is used 
in the last inequality.
\end{proof}

We will use Corollary \ref{corr:anncrossings} to prove 
that our cap process on $\BS^{d-1}$ is not connected whenever 
$\lambda>0$ is small enough. In order to do that, we will start by 
providing the following coupling result between our spherical 
cap process and the ball process in $\BR^{d-1}$ with intensity 
measure $\widetilde\mu^E_g.$ This will also deal with difficulty $(a)$ 
listed at the start of this section.

\begin{proposition} \label{prop:inclusion}
For measures $\widetilde\mu^{s}_g$ and $\widetilde\mu^E_g$ defined in \eqref{eqn:mucsdef}
and \eqref{eqn:mucEdef}, respectively, and using the same function $g,$
there exists a coupling of 
$\Psi(\lambda \widetilde\mu^{s}_g)$ and $\Psi(2 \lambda \widetilde\mu^E_g)$
such that almost surely,
\begin{align*}
&\left[-2C_d,2C_d\right]^{d-1} 
\cap \Pi_{\rm vert}\left(\CC(\lambda \widetilde\mu^{s}_g)\right)\subset \left[-2C_d,2C_d\right] \cap \CC(2 \lambda \widetilde\mu^E_g).
\end{align*}
\end{proposition}

\noindent
\begin{proof}
We start by considering the map 
$\varphi: \BS^{d-1}_+\times \left(0,C_d\right] 
\to \BR^{d-1} \times \left(0,C_d\right]$
defined by 
\[
\varphi((x_1,\ldots,x_d),\alpha):=((x_1,\ldots, x_{d-1}),\alpha)
=(\Pi_{\rm vert}(x_1,\ldots,x_d),\alpha).
\]
For $y\in B^{d-1}(o,1)$ we will let 
$\Pi_{\rm vert}^{-1}(y)$
be the unique point $x\in \BS_+^{d-1}$ such that $y=\Pi_{\rm vert}(x).$
We note that 
$\varphi^{-1}((y,\alpha))=(\Pi_{\rm vert}^{-1}(y),\alpha).$

Consider the image $\varphi(\widetilde\mu^{s}_g)$ of the measure $\widetilde\mu^{s}_g$ 
under the mapping $\varphi$, and let $\mu_{\varphi}^E$ be the restriction 
of this image to the set
$[-2C_d,2C_d]^{d-1}
\times (0,C_].$ Then, for any measurable 
$
A \subset \left[-2C_d,2C_d\right]^{d-1}$ and $0<\beta_1<\beta_2\leq C_d
$
we have 
\begin{equation}\label{eqn:tildenu}
\mu_{\varphi}^E(A\times (\beta_1,\beta_2))
=\widetilde\mu^{s}_g(\varphi^{-1}(A\times (\beta_1,\beta_2)))
=\sigma_{d-1}(\Pi_{\rm vert}^{-1}(A))
\int_{\beta_1}^{\beta_2} 
g(\alpha)\dint\alpha,
\end{equation} 
by the definition of $\widetilde\mu^s_{g}.$ We see that $\mu_\varphi^E$ is 
similar to $\widetilde\mu^E_g$ defined in \eqref{eqn:mucEdef}, but instead 
of $(d-1)$-dimensional 
Lebesgue measure we here have a
non-homogeneous measure with respect to the $x$-component (this is technical difficulty $(b)$ in the list at the start of this section). 
However, it is easy to see that
\begin{equation} \label{eqn:sigmaprojA}
\sigma_{d-1}(\Pi_{\rm vert}^{-1}(A))
=\int_A (1-\|x\|^2_{d-1})^{-{1\over 2}} \,\ell_{d-1}(\dint x),
\end{equation}
and it is immediate from \eqref{eqn:sigmaprojA} that 
$\ell_{d-1}(A)\leq \sigma_{d-1}(\Pi_{\rm vert}^{-1}(A))
\leq 2\ell_{d-1}(A)$ 
for every 
$A \subset \left[-2C_d,2C_d\right]^{d-1}$ (which is the reason why 
we chose to restrict our attention to the given geometry). 
Therefore, for any such $A$
and $(\beta_1,\beta_2) \subset 
(0,C_d]$ we conclude that
\begin{eqnarray*}
\lefteqn{\lambda \mu_\varphi^E(A\times(\beta_1,\beta_2))
=\lambda \sigma_{d-1}(\Pi_{\rm vert}^{-1}(A))
\int_{\beta_1}^{\beta_2} 
g(\alpha)\dint\alpha}\\
& & \leq 2 \lambda \ell_{d-1}(A) \int_{\beta_1}^{\beta_2} g(\alpha)\dint \alpha
=2\lambda \widetilde\mu^E_g(A\times(\beta_1,\beta_2)).
\end{eqnarray*}
Hence, as follows from the mapping properties of the Poisson  process (see \cite[Theorem 5.1]{LP18}) we 
can couple 
$\Psi(\lambda \mu_\varphi^E)$ and $\Psi(2 \lambda \widetilde\mu^E_g)$ so that almost surely
\begin{equation}\label{eqn:Phiinclusion}
\Psi(\lambda \mu_\varphi^E) \subset \Psi(2 \lambda \widetilde\mu^E_g),
\end{equation}
where we recall the convention explained in Section \ref{sec:Notation} that 
the Poisson  processes are regarded as sets.
Intuitively, this means that the projection onto $\BR^{d-1}$ of the 
points $x$ such that 
$(x,\alpha)\in \Psi(\widetilde\mu^s_{g})$ and 
$\Pi_{\rm vert}(x)\in\left[-2C_d,2C_d\right]^{d-1}$, are
dominated by a homogeneous process of intensity $2\lambda$ in 
$\left[-2C_d,2C_d\right]^{d-1}$. 

It remains to ensure
that the projection of the actual spherical cap $\scap(x,\alpha)$ is 
a subset of the ball $B^{d-1}(\Pi_{\rm vert}(x),\alpha).$ 
This can easily be done as follows. Fix $x\in \BS^{d-1}_+$ such that  
$\Pi_{\rm vert}(x)\in\left[-2C_d,2C_d\right]^{d-1}.$
We claim that
\begin{equation} \label{eqn:capballinclusion}
\Pi_{\rm vert}(\scap((x_1,\ldots,x_d),\alpha)) 
\subset B^{d-1}((x_1,\ldots, x_{d-1}),\alpha).
\end{equation}
To see this, first note that for any point $x'\in \scap(x,\alpha),$
we have that the $d$-dimensional Euclidean distance satisfies
$\Vert x-x'\Vert\leq \alpha.$ This holds since the angle
between the center $x$ and the points on the boundary of 
$\scap(x,\alpha)$ equals $\alpha.$ It follows that for any $x'\in \scap(x,\alpha)$
we have that 
\[
\Vert \Pi_{\rm vert}(x)-\Pi_{\rm vert}(x')\Vert
=\Vert(x_1,\ldots, x_{d-1})-(x_1',\ldots, x_{d-1}')\Vert_{d-1}\leq \alpha
\] 
and so \eqref{eqn:capballinclusion} holds. 

Next, we observe that for any $\scap(x,\alpha)$
where $x\in \BS^{d-1}_{+}$ is such that $(x_1,\ldots,x_{d-1}) \not \in \left[-2C_d,2C_d\right]^{d-1},$ 
it holds that 
$$
\Pi_{\rm vert}(\scap(x,\alpha))\cap \left[-2C_d,2C_d\right]^{d-1}=\emptyset.
$$
Indeed, this follows by \eqref{eqn:capballinclusion} since 
$\Pi_{\rm vert}(\scap(x,\alpha))$ can be inscribed in a ball 
$B^{d-1}(\Pi_{\rm vert}(x),\alpha)$  
and $\alpha\leq C_d.$ We therefore see that 
\begin{eqnarray} \label{eqn:proj1}
\lefteqn{ \left[-2C_d,2C_d\right]^{d-1} \cap \Pi_{\rm vert}\left(\CC(\widetilde\mu^{s}_g)\right)}\\
& & =
\left[-2C_d,2C_d\right]^{d-1} 
\cap \bigcup_{(x,\alpha)\in \Psi(\lambda \widetilde\mu^{s}_g)} 
\Pi_{\rm vert}(\scap(x,\alpha)) \nonumber \\
& & \subset  \left[-2C_d,2C_d\right]^{d-1} \cap 
\bigcup_{(x,\alpha)\in \Psi(\lambda\widetilde\mu^{s}_g)} 
B^{d-1}(\Pi_{\rm vert}(x),\alpha) \nonumber \\
& & =\left[-2C_d,2C_d\right]^{d-1} \cap 
\bigcup_{(y,\alpha)\in \varphi(\Psi(\lambda\widetilde\mu^{s}_g))} 
B^{d-1}(y,\alpha), \nonumber
\end{eqnarray}
where we used \eqref{eqn:capballinclusion} in the set-inequality.
Furthermore, by  definition of $\mu_\varphi^E$ we can 
couple $\Psi(\lambda\widetilde\mu^{s}_g)$ and $\Psi(\lambda \mu_\varphi^E)$ 
such that 
$\Psi(\lambda \mu_\varphi^E)=\varphi(\Psi(\lambda\widetilde\mu^{s}_g))$ almost surely.
We therefore conclude from \eqref{eqn:proj1} that 
\begin{eqnarray*}
\lefteqn{ \left[-2C_d,2C_d\right]^{d-1} \cap \Pi_{\rm vert}\left(\CC(\lambda \widetilde\mu^{s}_g)\right)}\\
& & \subset  \left[-2C_d,2C_d\right]^{d-1} \cap 
\bigcup_{(y,\alpha)\in \varphi(\Psi(\lambda\widetilde\mu^{s}_g))} 
B^{d-1}(y,\alpha) \\
& & =\left[-2C_d,2C_d\right]^{d-1} \cap 
\bigcup_{(y,\alpha)\in \Psi(\lambda \mu_\varphi^E)} 
B^{d-1}(y,\alpha) \\
& & \subset  \left[-2C_d,2C_d\right]^{d-1} \cap \bigcup_{(y,\alpha)\in \Psi(2 \lambda \widetilde\mu^E_g)} 
B^{d-1}(y,\alpha) \\
& & =\left[-2C_d,2C_d\right]^{d-1} \cap \CC(2 \lambda \widetilde\mu^E_g)
\end{eqnarray*}
with probability one, where we used \eqref{eqn:Phiinclusion} in the penultimate step.
\end{proof}

Finally, we are in a position to prove Theorem \ref{thm:disconnect2}.
\begin{proof}[Proof of Theorem \ref{thm:disconnect2}]
First we note that $\CC(\lambda \widetilde\mu^s_{g})$ is an open set since the 
caps used to define it are open. Therefore, connectness
is the same as path-connectness.

Assume first that $\CC(\lambda \widetilde\mu^s_{g})$ is connected. As in
Section \ref{sec:Notationfracball}, it is an easy consequence of 
\eqref{eqn:mucsdef} and the condition (stated below 
\eqref{eqn:mucsdef}) that 
$g(\alpha)\geq C' \alpha^{-d}$ for some $C'>0,$ 
that any point $x\in \BS^{d-1}$ is such that 
$\BP(x\in \CC(\lambda \widetilde\mu^s_{g}))=1.$ Therefore, connectedness of 
$\CC(\lambda \widetilde\mu^s_{g})$ implies the existence of a path 
$\omega\subset \CC(\lambda \widetilde\mu^s_{g})\subset \BS^{d-1}$ such that 
$\omega$ crosses the inverse projection of the annulus
$\CA(\frac{1}{4d},\frac{1}{2d}).$
That is, 
\[
\omega\subset \Pi_{\rm vert}^{-1}
\left(\CA\left(\frac{1}{4d},\frac{1}{2d}\right)\right)
\cap \CC(\lambda \widetilde\mu^s_{g}).
\]
Hence, it holds that 
$\Pi_{\rm vert}\left(\omega\right)$ crosses 
$\CA\left(\frac{1}{4d},\frac{1}{2d}\right)$ and
by using Proposition \ref{prop:inclusion} we see that 
there exists a coupling of $\Psi(\lambda \widetilde\mu^s_{g})$ and 
$\Psi(2 \lambda \widetilde\mu^E_g)$ such that 
$\Pi_{\rm vert}\left(\omega\right) 
\subset \CC(2 \lambda \widetilde\mu^E_g)$ almost surely. This follows since 
$\frac{1}{2d}\leq \frac{1}{2\sqrt{2(d-1)}}=C_d$ for every 
$d\geq 2.$ We conclude that the event 
\[
\Cr\left(\CA\left(\frac{1}{4d},\frac{1}{2d}\right),
\CC(2 \lambda \widetilde\mu^E_g)\right)
\] 
occurs, which, in turn, gives us that 
\begin{eqnarray*}
\lefteqn{\BP(\CC(\lambda \widetilde\mu^s_{g})\textrm{ is not connected})}\\
& & =1-\BP(\CC(\lambda \widetilde\mu^s_{g})\textrm{ is connected}) \\
& & \geq 1-\BP\left(\Cr\left(\CA\left(\frac{1}{4d},\frac{1}{2d}\right),
\CC(2 \lambda \widetilde\mu^E_g)\right)\right)>0,
\end{eqnarray*}
where we used Corollary \ref{corr:anncrossings} in the last inequality.
\end{proof}

%\section*{Acknowledgements}
%
%AG and CT have been supported by the DFG priority program SPP 2265 \textit{Random Geometric Systems}. CT was also supported by the DFG priority program SPP 2458 \textit{Combinatorial Synergies}. AG was also supported by the DFG under Germany's Excellence Strategy  EXC 2044 -- 390685587, \textit{Mathematics M\"unster: Dynamics - Geometry - Structure}.


\begin{thebibliography}{99}

\bibitem{AhlbergTassioTeiceira}
Ahlberg, D., Tassion, V. and Teixeira, A.: Sharpness of the phase transition for continuum percolation in $\mathbb{R}^2$. {\em Probab. Theory Related Fields} \textbf{172}, 525--581 (2018).

\bibitem{BenjaminiSchramm}
Benjamini, I. and Schramm, O.: Percolation in the hyperbolic plane. {\em J. Amer. Math. Soc.} \textbf{14}, 487--507 (2001).

\bibitem{BenjaminiVisibility}
Benjamini, I., Jonasson, J., Schramm, O. and Tykesson, J.: Visibility to infinity in the hyperbolic plane, despite obstacles. {\em ALEA Lat. Am. J. Probab. Math. Stat.} \textbf{6}, 323--342 (2009).


\bibitem{Broman} Broman, E.: The existence phase transition for scale invariant Poisson random fractal models. 
{\em Ann. Inst. Henri Poincaré Probab. Stat.} {\bf 56},
715--733 (2020). 

\bibitem{BC} Broman E. and Camia F.:
Universal behavior of connectivity properties in fractal percolation models.
{\em Electron. J. Probab.} {\bf 15}, 1394--1414 (2010).
	
\bibitem{BromanTykessonCylHyperbolic}	
Broman, E. and Tykesson, J.: Poisson cylinders in hyperbolic space. {\em Electron J. Probab.} {\bf 20}, paper 41, 25pp.\ (2015).

\bibitem{DuminilRaouflTassio}
Duminil-Copin, H., Raoufi, A. and Tassion, V.: Subcritical phase of $d$-dimensional Poisson-Boolean percolation and its vacant set. {\em Ann. H. Lebesgue} \textbf{3}, 677--700 (2020).

\bibitem{GuereMarchand18}
Gouéré, J.-B. and Marchand, R.: Continuum percolation in high dimensions.
{\em Ann. Inst. Henri Poincaré Probab. Stat. }  \textbf{54}, 1778--1804 (2018).

\bibitem{GuereLabey}
Gouéré, J.-B. and Layéy, F.: Percolation in the Boolean model with convex grains in high dimension. {\em Electron. J. Probab.} \textbf{28}, paper 111, 57pp. (2023).


\bibitem{LP18}
Last, G. and Penrose, M.: {\em Lectures on the Poisson Process}. Cambridge University Press (2018).


\bibitem{Liggett} Liggett, T. M.: \textit{Stochastic Interacting Systems: Contact, Voter and Exclusion Processes.} Springer (1999).

\bibitem{MeesterRoy}
Meester, R. and Roy, R.: {\em Continuum Percolation}. Cambridge University Press (1996).

\bibitem{GS05} Gallego, E. and Solanes, G.: Integral geometry and geometric inequalities in hyperbolic space. {\em Differential Geom. Appl.}, 
{\bf 22}, 315--325 (2005).

\bibitem{O96} Orzechowski M. E.: One the phase transition to sheet percolation in  random Cantor sets, {\em J. Statist. Phys.} {\bf 82}, 1081--1098 (1996).

\bibitem{Santalo}
Santal\'o, L. A.: {\em Integral Geometry and Geometric Probability}. Addison-Wesley (1976). 

\bibitem{SW}
Schneider, R. and Weil, W.: {\em Stochastic and Integral Geometry}. Springer (2008).

\bibitem{Tykesson07}
Tykesson, J.: The number of unbounded components in the Poisson Boolean model of continuum percolation in hyperbolic space. {\em Electron. J. Probab.} \textbf{12}, 1379--1401 (2007).

\bibitem{Wyner}
 Wyner, A. D.: Random packings and coverings of the unit $ n $-sphere. {\em Bell System Tech. J.} {\bf 46},  2111--2118 (1967).

\end{thebibliography}
\end{document}